\documentclass[final,3p,times]{elsarticle}

\usepackage{amssymb}
\usepackage{graphicx,setspace}
\usepackage{psfrag}
\usepackage{amsfonts}
\expandafter\let\csname equation*\endcsname\relax
\expandafter\let\csname endequation*\endcsname\relax
\usepackage{amsmath}
\usepackage{amssymb, amsthm}
\usepackage{mathrsfs}
\usepackage{epstopdf}
\usepackage{float}
\usepackage{lineno,hyperref}
\usepackage{color}
\usepackage{subfigure}

\usepackage{amsmath}
\usepackage{dsfont} 
\usepackage{amssymb} 
\usepackage{graphicx,color}
\usepackage{extarrows}

\usepackage{graphicx}

\usepackage{epstopdf}

\journal{SIAM}

\begin{document}
\newtheorem{definition}{Definition}[section]
\newtheorem{lemma}{Lemma}[section]
\newtheorem{remark}{Remark}[section]
\newtheorem{theorem}{Theorem}[section]
\newtheorem{proposition}{Proposition}
\newtheorem{assumption}{Assumption}
\newtheorem{example}{Example}
\newtheorem{corollary}{Corollary}[section]
\def\ep{\varepsilon}
\def\Rn{\mathbb{R}^{n}}
\def\Rm{\mathbb{R}^{m}}
\def\E{\mathbb{E}}
\def\hte{\hat\theta}
\renewcommand{\theequation}{\thesection.\arabic{equation}}
\begin{frontmatter}



\title{Modulation and amplitude equations on bounded domains for nonlinear SPDEs driven by cylindrical $\alpha$-stable L\'evy processes}

\author{Shenglan Yuan\fnref{addr1}}\ead{shenglan.yuan@math.uni-augsburg.de}
\author{Dirk Bl$\rm\ddot{o}$mker\corref{cor1}\fnref{addr1}}
\ead{dirk.bloemker@math.uni-augsburg.de}\cortext[cor1]{Corresponding author}

\address[addr1]{\rm Institut f$\rm\ddot{u}$r Mathematik, Universit$\rm\ddot{a}$t Augsburg,
86135, Augsburg, Germany }

\begin{abstract}
In the present work, we establish the approximation of nonlinear stochastic partial differential
equation (SPDE) driven by cylindrical $\alpha$-stable L\'evy processes via modulation or amplitude equations.

We study SPDEs with a cubic nonlinearity, where the deterministic equation
is close to a change of stability of the trivial solution.
The natural separation of time-scales close to this bifurcation allows us to obtain an amplitude equation
describing the essential dynamics of the bifurcating pattern, thus
reducing the original infinite dimensional dynamics to a simpler finite-dimensional effective dynamics.
In the presence of a multiplicative stable L\'evy noise
that preserves the constant trivial solution
we study the impact of noise on the approximation.

In contrast to Gaussian noise, where non-dominant pattern are uniformly small in time due to averaging effects, large jumps in the L\'evy noise might lead to large error terms, and thus new estimates are needed to take this into account.
\end{abstract}


\end{frontmatter}


\section{Introduction}

Models of modulation  or amplitude equations \cite{BB,BHP,MBK,XKK} have proven to be rather universal and efficient in describing the dynamics associated with a qualitative change of stability (bifurcation). Such structures emerge in fields ranging from spatially and temporarily oscillating wave packets \cite{Osb} to long waves in dispersive media \cite{TT} and spatio-temporal pattern in dissipative systems \cite{SS}. In particular, the Ginzburg-Landau equation plays a prominent role as the effective modulation equation for the description of pattern forming systems close to the first instability since the 1960s; see Newell and Whitehead \cite{NW}. Among these, arguably, the most prototypical one is the Allen-Cahn equation with bistable behavior \cite{ER}, which characterizes interface motion between two stable phases.

The mathematical justification of modulation equation
beyond pure formal calculations
has been started in the early 90th, see for example \cite{Mi92,CE90,KSM92,Schn94a,Schn94c}.
All these results and many later     treated the case of unbounded domains,
as in the bounded doamin case the theory of
center manifolds is available in order to reduce the dynamics,
which does not help in the stochastic case.

Earlier works in the stochastic case
studied almost all the case of Gaussian additive noise,
and starting from \cite{BMS} many articles explored the use of amplitude
approximation in order to qualitatively examine the dynamics of stochastic systems near a change of stability.
The quantitative error estimates are usually done pathwise with high probability  or in moments on the natural slow time-scales close to the bifurcation.
Nevertheless modulation or amplitude equations can also give the approximation in the long-time behavior, for example the approximation of the infinite-dimensional invariant measure for a Swift-Hohenberg equation \cite{BH}.
Also ideas presented in \cite{B07,BH05} can be used
in approximating random attractors or random invariant manifolds via amplitude  equations.

The case of multiplicative Wiener noise is not that well studied, but also here amplitude equations provide insights into the impact of multiplicative noise in SPDEs close to bifurcation; see \cite{DH} for an example and \cite{B07} for general results on scalar one-dimensional noise.

It is worthwhile to note that here we consider our SPDEs on a bounded domain only,
thus leading to a finite dimensional space of dominant patterns that change their stability. In this setting, the amplitude equation
turns out to be a stochastic ordinary differential equation (SDE) describing the amplitude of these dominant modes.
For the case of SPDEs on an unbounded domain,
the effective equation is no longer an SDE and the amplitude of a dominant mode is slowly modulated in space,
thus the reduced model is still an infinite dimensional SPDE.
Nevertheless, we will not focus on this case here.
See \cite{BBS} for the full approximation of Swift-Hohenberg perturbed by space-time white noise on the whole real line, \cite{MBK} in the case of a simple one-dimensional noise, and \cite{BHP} for large domain.

In this paper, we study the following class of stochastic partial differential equations (SPDEs) driven by cylindrical $\alpha$-stable L\'evy process of the following form,
\begin{equation}\label{GinzburgLandau}
du(t)=[\mathcal{A}u(t)+\varepsilon^{2}\mathcal{L}u(t)+\mathcal{F}(u(t))]dt+\varepsilon^{\frac{2}{\alpha}}G(u(t))dL^{\alpha}(t),
\end{equation}
where $\mathcal{A}$ is a non-positive self-adjoint operator with finite-dimensional kernel, $\varepsilon^{2}\mathcal{L}$ represents a small deterministic perturbation with a small parameter $\varepsilon>0$ measuring the distance to bifurcation (the change of stability). The nonlinearity $\mathcal{F}$ stands for a cubic mapping where one standard example is the cubic nonlinearity $-u^3$,
and $G(u)$ denotes a Hilbert-Schmidt operator with $G(0)=0$ so that the constant $u=0$ is a solution to equation (\ref{GinzburgLandau}).
The prototypical case is the multiplication with $u$ combined with a fixed Hilbert-Schmidt operator independent of $u$ that regularizes the noise.
The noise $L^{\alpha}(t)$ is a cylindrical $\alpha$-stable L\'evy process on some stochastic basis with the index $\alpha\in(1,2)$ of stability.
We will give more details on the setting in our assumptions below.

Our aim in the present work is to explore the asymptotic dynamics
in the limit $\varepsilon\rightarrow0$ of solutions $u(t)$ to equation \eqref{GinzburgLandau} on the natural slow time-scale of order $\varepsilon^{-2}$.
Utilizing a separation of time-scales, near a change of stability for the linearized operator $\mathcal{A}+\varepsilon^{2}\mathcal{L}$, the system (\ref{GinzburgLandau}) can be transformed to the slow dynamics where the dominant
pattern is still coupled to the dynamics on a fast time scale.
A reduced equation eliminating the
fast variable and characterizing the behaviour of dominant modes significantly simplifies the
dynamics to an SDE, which we classify as amplitude equation identifying the essential dynamics of
dominant pattern.

The scaling of the $\alpha$-stable noise is chosen in such a way
that it has an impact on the slow time-scale. If we take a larger exponent in the noise strength $\varepsilon^{2/\alpha}$, then we expect the noise to have no impact on the approximation, while for smaller exponents the noise should dominate it and we loose the impact of $\mathcal{L}$.
Another equivalent point of view is, that with a fixed noise strength we have to choose
exactly the right distance from bifurcation in order to see an impact of the small noise on the bifurcation.

The main advantage of $\alpha$-stable noise is that as in the Gaussian case it is a self-similar process that scales in time, so we can rescale equations to the slow time-scale easily.
The disadvantage are the large jumps. These lead to large error terms and we are not able to use uniform error bounds in time as in the Gaussian case.

Previous approximation results via amplitude equations considered mainly square integrable processes or even Gaussians having all moments.
This rules out the interesting  cylindrical $\alpha$-stable  L\'evy process that only has finite $p$th moment for $p\in (0, \alpha)$.
Many  tools developed so far are not suitable to treat cylindrical $\alpha$-stable  L\'evy noises with the loss of the second moment,  such as  Kunita's inequality, Burkholder-Davis-Gundy inequality and Da
Prato-Kwapie\'n-Zabczyk's factorization technique \cite{PZ92,SZ}.
Therefore, we require new and different techniques to explore the cylindrical $\alpha$-stable noise more carefully. A challenging problem in this paper is how to handle the nonlinear terms, where the techniques of stopping times is used frequently
in order to cut-off the nonlinear terms that get too large.
But in connection with large jumps induced by
the cylindrical $\alpha$-stable L\'evy noise this causes many technical problems
which we had to overcome.

Our presentation is structured as follows. In Section 2, we  briefly present the theoretical assumptions and  analysis tools of the estimates for main results. While  Section 3 provides our main results for considering the amplitude equation of equation (\ref{GinzburgLandau}). In Section 4, we analyze examples to illustrate applications of our main results.
Finally, Section 5 summarizes our findings and showcase our conclusions, as well as a number of directions for future study.

%
\section{Assumptions and analysis tools}
%

Throughout the paper, we shall work in a separable Hilbert space $\mathcal{H}$, endowed with the usual
scalar product $\langle\cdot,\cdot\rangle$ and with the corresponding norm $\| \cdot \|$.
For any $\theta\in \mathbb{R}$, by using the domain of definition for fractional powers of the operator $\mathcal{A}$,
\begin{equation*}
\mathcal{H}^{\theta}:=D((1-\mathcal{A})^{\theta}):=\Big\{h=\sum_{k=1}^{\infty}h_{k}e_{k}: h_k=\langle h,e_k\rangle\in \mathbb{R}, \sum_{k=1}^{\infty}(\lambda_{k}+1)^{2\theta}h_{k}^{2}<\infty\Big\},
\end{equation*}
where $e_k$ is an orthonormal basis of eigenfunctions
such that $-\mathcal{A}e_k=\lambda_ke_k$  and
\begin{equation*}
(1-\mathcal{A})^{\theta}h:=\sum_{k=1}^{\infty}(\lambda_{k}+1)^{\theta}h_ke_k,~~~h\in D((1-\mathcal{A})^{\theta}),
\end{equation*}
with the associated norm
\begin{equation*}
\| h\|_{\theta}=\| \sum_{k=1}^{\infty}h_ke_k\|_{\theta}:=\|(1-\mathcal{A})^{\theta}h\|:=\sqrt{\sum_{k=1}^{\infty}(\lambda_{k}+1)^{2\theta}h_{k}^{2}} ,
\end{equation*}
where $:=$ will be used hereafter to denote definitions. It is straightforward to infer that $\mathcal{H}^{0}=\mathcal{H}$, $\mathcal{H}^{1}=D((1-\mathcal{A}))$ and $\mathcal{H}^{-\theta}$ is the dual space of $\mathcal{H}^{\theta}$.

The cylindrical $\alpha$-stable L\'evy process $L^{\alpha}(t)$ is defined via
\begin{equation*}
L^{\alpha}(t)=\sum_{k=1}^{\infty}L_{k}^{\alpha}(t)e_{k},~~~~~~t\geq0,
\end{equation*}
where  $\{L_{k}^{\alpha}(t)\}_{k=1}^{\infty}$ are independent one  dimensional $\alpha$-stable L\'evy processes on stochastic base $(\Omega,\mathcal{F},\{\mathcal{F}_t\}_{t\geq0},\mathbb{P})$. 
They are purely jump L\'evy processes and have the same characteristic function by L\'evy-Khinchine formula \cite{BSW}, i.e., $\mathbb{E}[e^{i\xi L_{k}^{\alpha}(t)}]=e^{t\psi(\xi)}, t\geq0, k\in\mathbb{N}^{\ast}$,
where $\psi(\xi)$ is the L\'evy symbol given by
$$
\psi(\xi)=-|\xi|^{\alpha}=\int_{{\mathbb{R}}\setminus{\{0\}}}(e^{i\xi y}-1-i\xi y\mathds{1}_{\{|y|<1\}})\nu_{\alpha}(dy).
$$
Here $\nu_{\alpha}$ is the L\'evy measure satisfying $\int_{{R}\setminus{\{0\}}}1\wedge|y|^{2}\nu_{\alpha}(dy)<\infty$, which is determined by
 $$
 \nu_{\alpha}(dy)=c(1,{\alpha})\frac{1}{|y|^{1+\alpha}}dy,
 $$
 where $ c(1,\alpha)=\frac{\alpha\Gamma(\frac{1+\alpha}{2})}{2^{1-\alpha}\sqrt{\pi}\Gamma{(1-\frac{\alpha}{2})}}$ and $\Gamma$ is the Gamma function. For $t>0$ and Borel set $B\in\mathcal{B}({\mathbb{R}}\setminus{\{0\}})$, define the Poisson random measure of $L_{k}^{\alpha}(t)$ by
 \begin{equation*}
 N_{k}(t,B)=\sum_{0< s\leq t}\mathds{1}_{B}(L_{k}^{\alpha}(s)-L_{k}^{\alpha}(s-))=\#\{s\in(0,t]: L_{k}^{\alpha}(s)-L_{k}^{\alpha}(s-)\in B\},
 \end{equation*}
where $L_{k}^{\alpha}(s-)$ is the left limit of $L_{k}^{\alpha}(s)$. The function $\nu_{\alpha}(B)=\mathbb{E}(N(1,B))$ of the L\'evy measure is to describe the expected number of jumps in a certain size at a time interval $(0,1]$.
Furthermore, define the  compensated Poisson measure of $L_{k}^{\alpha}(t)$ via
\begin{equation*}
\tilde{N}_k(t,B)=N_k(t,B)-t\nu_{\alpha}(B).
\end{equation*}
According to the L\'evy-It\^o decomposition \cite{ST}, $L_k^{\alpha}(t)$ are able to be expressed as
\begin{equation}\label{Ito-Dec}
 L_k^{\alpha}(t)=\int_{|y|<1}y\tilde{N}_k(t,dy)+\int_{|y|\geq 1}yN_k(t,dy).
\end{equation}



In order to study system \eqref{GinzburgLandau}, we impose the following assumptions.

\noindent $(\textbf{A1}).$
(Linear operator $\mathcal{A}$ ) Assume that the leading operator $\mathcal{A}$ is a self-adjoint and
non-positive operator on $\mathcal{H}$ with eigenvalues $\{- \lambda_{k}\}_{k=1}^{\infty}$ such that $0 =\lambda_{1}\leq...\leq\lambda_{k}...$, satisfying $\lambda_{k} \to \infty$ for $k\to \infty$. The eigenvectors $\{e_{k}\}_{k=1}^{\infty}$ of $\mathcal{A}$ form a complete orthonormal basis in $\mathcal{H}$ such that $\mathcal{A}e_{k}= - \lambda_{k}e_{k}$.

Denote the kernel space of $\mathcal{A}$ by $\mathcal{N}:=\text{ker}(\mathcal{A})$. According to assumption $(\textbf{A1})$, $\mathcal{N}$ has finite dimension $n$ with basis ${e_1,...,e_n}$, i.e.,
$\mathcal{N}=\text{span}\{e_1,...,e_n\}$,
which means $\lambda_{n}=0<\lambda_{n+1}$.
By $\mathcal{P}_c$ we denote the orthogonal projector from $\mathcal{H}$ onto $\mathcal{N}$
with respect to the inner product $\langle\cdot,\cdot\rangle$ , and by $\mathcal{P}_s:=\mathcal{I}-\mathcal{P}_c$ the orthogonal projector from $\mathcal{H}$ onto the orthogonal complement $\mathcal{S}=\mathcal{N}^{\perp}$,  where $\mathcal{I}$ is the
identity operator on $\mathcal{H}$. For shorthand notation, we use the subscripts $c$ and $s$ for projection onto $\mathcal{N}$ and $\mathcal{S}$, i.e., $\mathcal{A}_c:=\mathcal{P}_c\mathcal{A}$ and $\mathcal{A}_s:=P_s\mathcal{A}$. We define $\mathcal{L}_c$, $\mathcal{L}_s$, $\mathcal{F}_c$ and $\mathcal{F}_s$ in a similar way.

\noindent $(\textbf{A2}).$
(Operator $\mathcal{L}$ ) Let $\mathcal{L}:\mathcal{H}^{\theta}\rightarrow\mathcal{H}^{\theta-\sigma}$ for some
$\theta\in\mathbb{R}$, $\sigma\in[0,1)$ be a linear continuous mapping that commutes with $\mathcal{P}_c$ and $\mathcal{P}_s$.

This assumption is crucial for our approach. If we do not assume that $\mathcal{P}_c$ and $\mathcal{P}_s$ commute with $\mathcal{L}$, then we expect an additional linear coupling of $a$ and $b$ in our formal calculation below, which changes the result completely.

\noindent $(\textbf{A3}).$
(Nonlinearity $\mathcal{F}$). Suppose that $\mathcal{F}$: $(\mathcal{H}^{\theta})^{3}\rightarrow\mathcal{H}^{\theta-\sigma}$,
with $\theta\in\mathbb{R}$, $\sigma\in[0,1)$ from (A2) is a trilinear, symmetric mapping and satisfies the following conditions. For some $C>0$,
\begin{equation}\label{F}
\|\mathcal{F}(u,v,w)\|_{\theta-\sigma}\leq C\|u\|_{\theta}\|v\|_{\theta}\|w\|_{\theta}~~~~\text{for all}~~ u, v, w\in\mathcal{H}^{\theta}.
\end{equation}
Moreover, we have on the space $\mathcal{N}$ the stronger assumptions
\begin{align}\label{fc}
\langle\mathcal{F}_{c}(u),u\rangle&\leq0~~~~\text{for all}~~ u\in\mathcal{N},\\ \label{fu}
\langle\mathcal{F}_{c}(u,u,w),w\rangle&\leq0~~~~\text{for all}~~ u, w\in\mathcal{N},
\end{align}
and for some positive constants $C_{0}$, $C_{1}$ and $C_{2}$,
\begin{equation}\label{fvw}
\langle\mathcal{F}_{c}(u,v,w)-\mathcal{F}_{c}(v),u\rangle\leq-C_{0}\|u\|^{4}+C_{1}\|w\|^{4}+C_{2}\|w\|^{2}\|v\|^{2}~~~~\text{for all}~~ u, v, w\in\mathcal{N}.
\end{equation}
To ease notation, we use $\mathcal{F}(u)=\mathcal{F}(u,u,u)$ for shorthand notation throughout the paper.

Let $L_{HS}(\mathcal{H},\mathcal{H}^{\theta})$ denote the space consisting of all Hilbert-Schmidt
operators from $\mathcal{H}$ to $\mathcal{H}^{\theta}$, where
the norm is given by $\|\Psi\|^2_{L_{HS}}=\sum_{k=1}^\infty\|\Psi e_k\|^2_{\theta}$
for any orthonormal basis $\{e_k\}_{k\in\mathbb{N}}$ of $\mathcal{H}$.

\noindent $(\textbf{A4}).$
(Operator $G$)  Assume that
$G:\mathcal{H}^{\theta}\rightarrow L_{HS}(\mathcal{H},\mathcal{H}^{\theta})$ satisfying $G(0)=0$, with $\theta\in\mathbb{R}$ from (A2) and (A3), is Fr\'echet differentiable up to order $2$ and fulfills the following conditions.
For one $r>0$, there exists a constant $l_r>0$ such that for all $u, v, w\in \mathcal{H}^{\theta}$ with $\|u\|_{\theta}\leq r$,
\begin{align}\label{g}
\|G(u)\|_{L_{HS}}&\leq l_r\|u\|_{\theta},\\ \label{gv}
\|G'(u)\cdot v\|_{L_{HS}}&\leq l_r\|v\|_{\theta},
\end{align}
and
\begin{equation}\label{gvw}
\|G''(u)\cdot(v,w)\|_{L_{HS}}\leq l_r\|v\|_{\theta}\|w\|_{\theta},
\end{equation}
where the notations $G'(u)$ and $G''(u)$ denote the first
and second Fr\'echet derivatives at point $u$, respectively.

We need to control the convergence of various infinite series, which is possible if the
noise is not too irregular.

\noindent
$(\textbf{A5}).$  Define $\beta_k$ so that
 $\|\mathcal{P}_c[G'(0)v] e_k \| \leq \beta_k\|v\|$ for all $v\in  \mathcal{H}^{\theta}$ with $\|u\|_{\theta}\leq r$
 then we assume that $\sum_{k=1}^\infty \beta_k <\infty$.

This means that $\beta_{k}$ decays sufficiently fast when $k\rightarrow\infty$. The assumption  $\sum_{k=1}^\infty \beta_k <\infty$ is stronger than Hilbert-Schmidt, which would be $\sum_{k=1}^\infty \beta_k^2 <\infty$.



It is well known that $\mathcal{A}$ is the infinitesimal generator of an analytic semigroup $\{e^{t\mathcal{A}}\}_{t\geq0}$ on $\mathcal{H}^{\theta}$ with
\begin{equation*}
e^{t\mathcal{A}}\Big(\sum_{k=1}^{\infty}h_ke_k\Big)=\sum_{k=1}^{\infty}e^{- \lambda_{k}t}h_ke_k,~~~~t\geq0.
\end{equation*}
Then  we have the following useful estimate. It is a classical property for an analytic semigroup and we omit the proof.

\begin{lemma}\label{add}
Under assumption $(\textbf{A1})$, for all $\sigma\leq\theta$, $\rho\in(\lambda_{n},\lambda_{n+1})$, $n\in\mathbb{N}^{\ast}$, there exists a constant $M > 0$, which is independent of $h\in\mathcal{H}$, such that for any $t>0$,
\begin{equation}\label{es}
\| e^{t\mathcal{A}}\mathcal{P}_s h\|_{\theta}\leq Mt^{-(\theta-\sigma)}e^{-\rho t}\| \mathcal{P}_s h\|_{\sigma}.
\end{equation}
\end{lemma}

To give a meaning to the solution of system \eqref{GinzburgLandau}, we use the definition of local mild solution as in \cite{PZ92}.

\begin{definition}\label{mild}
(Local mild solution).
An $\mathcal{H}^{\theta}$-valued stochastic process $\{u(t)\}_{t\in[0,T]}$, is called a local
mild solution of equation \eqref{GinzburgLandau} if for some stopping time $\tau_{ex}>0$ we have on a set of probability $1$
that $u\in\mathcal{D}([0,\tau_{ex}),\mathcal{H}^{\theta})$ (i.e., a process with c\'adl\'ag paths) and
\begin{equation*}
u(t)=e^{t\mathcal{A}}u(0)+\int_{0}^{t}e^{(t-s)\mathcal{A}}[\varepsilon^{2}\mathcal{L}u(s)+\mathcal{F}(u(s))]ds+\varepsilon^{\frac{2}{\alpha}}\int_{0}^{t}e^{(t-s)\mathcal{A}}G(u(s))dL^{\alpha}(s)
\end{equation*}
for all $t\in(0,\tau_{ex})$.

Moreover, $\tau_{ex}$ is maximal, which means that $\mathbb{P}$-almost surely $\tau_{ex} =\infty$ or $\lim_{t\nearrow\tau_{ex}} \|u(t)\|_\theta=\infty.$
\end{definition}
\begin{remark}
The proof of the existence and uniqueness of a local mild solution
should be fairly standard under our assumptions,
using a cut-off of the nonlinearity so that the nonlinearities are globally Lipschitz together with a fixed-point argument.
Although this is not present in the literature,
we will not go into details in this paper.
For simplicity, ew always assume that we have a mild solution
in the sense of Definition \eqref{mild}.
\end{remark}

Let
\begin{equation}\label{put}
\Psi(t)=\Psi_{0}\mathds{1}_{\{0\}}(t)+\sum_{k=1}^{N-1}\Psi_{k}\mathds{1}_{(t_{k},t_{k+1}]}(t)
\end{equation}
be a simple stochastic process, where $0=t_{1}<...<t_{N}=T$, and $\Psi_{k}$ are $\mathcal{F}_{t_{k}}$-measurable $L_{HS}(\mathcal{H},\mathcal{H}^{\theta})$-valued random variables. We assume that $\Psi$ is predictable in the sense that
for $k\in\{0,1,...,N-1\}$ and $h\in\mathcal{H}^{\theta}$, $\Psi_{k}h$ are $\mathcal{F}_{t_{k}}$-measurable $\mathcal{H}^{\theta}$-valued random variables. Write the stochastic integral with respect to the $\alpha$-stable cylindrical L\'evy process $L^{\alpha}$:
\begin{equation*}
\int_{0}^{t}\Psi(s)dL^{\alpha}(s)=\sum_{k=1}^{N-1}\Psi_{k}(L^{\alpha}(t_{k+1}\wedge t)-L^{\alpha}(t_{k}\wedge t)),
\end{equation*}
which is not continuous \cite{KR}.

To extend the definition of the stochastic integral to more general processes, it is convenient to regard integrands as random variables defined on the product space $[0,T]\times\Omega$, equipped with the product $\sigma$-algebra $\mathcal{B}([0,T])\times\mathcal{F}$. The product measure of the Lebesgue measure $dt$ on $[0,T]$ and the probability measure $\mathbb{P}$ is represented by $\mathbb{P}_{T}:=dt\times\mathbb{P}$. Let $\Lambda(L_{HS}(\mathcal{H},\mathcal{H}^{\theta}))$ denote the space of predictable processes $\Psi:[0,T]\times\Omega\rightarrow L_{HS}(\mathcal{H},\mathcal{H}^{\theta})$ such that
\begin{equation*}
\|\Psi\|_{\Lambda}:=\mathbb{E}\Big(\int_{0}^{T}\|\Psi(s)\|_{L_{HS}}^{\alpha}ds\Big)^{\frac{1}{\alpha}}<\infty,~~~~~~\alpha\in(1,2).
\end{equation*}
That is, $\Lambda(L_{HS}(\mathcal{H},\mathcal{H}^{\theta}))=L^{\alpha}([0,T]\times\Omega,\mathcal{B}([0,T])\times\mathcal{F},\mathbb{P}_{T};L_{HS}(\mathcal{H},\mathcal{H}^{\theta}))$.
We denote with $\Lambda^{s}(L_{HS}(\mathcal{H},\mathcal{H}^{\theta}))$ the space of simple processes of the form \eqref{put}. We follow the approach of \cite[Proposition 4.22(ii)]{PZ92} to verify that the space $\Lambda^{s}(L_{HS}(\mathcal{H},\mathcal{H}^{\theta}))$ is dense in
the space $\Lambda(L_{HS}(\mathcal{H},\mathcal{H}^{\theta}))$ with respect to the norm $\|\cdot\|_{\Lambda}$. Regarding this, we have the following proposition.
\begin{proposition}\label{Prop}
If $\Psi$ is a process belonging to $\Lambda(L_{HS}(\mathcal{H},\mathcal{H}^{\theta}))$, then there exists a sequence $\{\Psi_{n}\}$ of simple processes belonging to $\Lambda^{s}(L_{HS}(\mathcal{H},\mathcal{H}^{\theta}))$ such that $\|\Psi-\Psi_{n}\|_{\Lambda}\to 0$ as $n\to \infty$.
\end{proposition}
\begin{proof}
Since the space $L_{HS}(\mathcal{H},\mathcal{H}^{\theta})$ is densely embedded into $\Lambda(L_{HS}(\mathcal{H},\mathcal{H}^{\theta}))$, there exists a sequence $\{\Psi_{n}\}$ of $L_{HS}(\mathcal{H},\mathcal{H}^{\theta})$-valued predictable simple processes
\begin{equation*}
\Psi_{n}(t)=\Psi_{0,n}\mathds{1}_{\{0\}}(t)+\sum_{k=1}^{N-1}\Psi_{k,n}\mathds{1}_{(t_{k},t_{k+1}]}(t)
\end{equation*}
 on $[0,T]$ taking on only
a finite numbers of values such that
\begin{equation*}
\|\Psi(t,\omega)-\Psi_{n}(t,\omega)\|_{L_{HS}}\downarrow0
\end{equation*}
for all $(t,\omega)\in[0,T]\times\Omega$. Consequently $\|\Psi-\Psi_{n}\|_{\Lambda}\downarrow0$. Therefore it is sufficient to prove that for arbitrary $A\in\mathcal{B}([0,T])\times\mathcal{F}$ and arbitrary $\varepsilon>0$ there exists a finite sum $B$ of disjoint sets of the form
\begin{equation}\label{form}
(s,t]\times F,~~0\leq s<t\leq T,~~F\in\mathcal{F}_{s}~~~~\text{and}~~~~\{0\}\times F,~~F\in\mathcal{F}_{0},
\end{equation}
such that
\begin{equation*}
\mathbb{P}_{T}\{(A\setminus B)\cup(B\setminus A)\}<\varepsilon.
\end{equation*}
To show this let us denote by $\mathcal{K}$ the family of all finite sums for sets of the form \eqref{form}.
Because $\varnothing\in\mathcal{K}$ and if $B_{1}, B_{2}\in\mathcal{K}$ then $B_{1}\cap B_{2}\in\mathcal{K}$, $\mathcal{K}$ is a
$\pi$-system. Let $\mathcal{G}$ be the family of all $A\in\mathcal{B}([0,T])\times\mathcal{F}$ which can be approximated in the above sense by elements from $\mathcal{K}$. One can check that $\mathcal{K}\subset\mathcal{G}$ and if $A\in\mathcal{G}$ then the complement $A^{c}$ satisfies $A^{c}\in\mathcal{G}$, and that if $A_{k}\in\mathcal{G}$ for all $k\in\mathbb{N}$ and $A_{n}\cap A_{m}=\varnothing$ for $n\neq m$, then $\cup_{k=1}^{\infty}A_{k}\in\mathcal{G}$. Hence $\sigma(\mathcal{K})=\mathcal{B}([0,T])\times\mathcal{F}=\mathcal{G}$ are required.
\end{proof}


\begin{remark}
 Using Proposition \ref{Prop}, we are able to extend the definition of stochastic integral $\int_{0}^{t}\Psi(s)dL^{\alpha}(s)$ to all $\Lambda(L_{HS}(\mathcal{H},\mathcal{H}^{\theta}))$-predictable processes $\Psi$.
\end{remark}

A moment inequality was proven for $\alpha$-stable L\'evy process in the case of real-valued integrand and
vector-valued integrator from \cite[Theorem 4.2]{RW}. When $p<\alpha$,
\begin{equation}\label{moment}
\mathbb{E}\Big(\sup_{0\leq t\leq T}\|\int_{0}^{t}F(s)dL^{\alpha}(s)\|\Big)^{p}\leq C\mathbb{E}\Big(\int_{0}^{T}|F(s)|^{\alpha}ds\Big)^{\frac{p}{\alpha}},
\end{equation}
where real process $F\in L^{q}([0,T]), 0<q<2$.
Similarly, one can obtain the following
moment inequality. See \cite{Rnew} and also \cite{R} for the $L^2$-theory.
\begin{equation*}
\mathbb{E}\Big(\sup_{0\leq t\leq T}\|\int_{0}^{t}\Psi(s)dL^{\alpha}(s)\|_{\theta}\Big)^{p}\leq C\mathbb{E}\Big(\int_{0}^{T}\|\Psi(s)\|_{L_{HS}}^{\alpha}ds\Big)^{\frac{p}{\alpha}}.
\end{equation*}
\begin{remark}
It should be highlighted here that we can take $\Psi(s)=G(u(s))$ to introduce one of our main tools
\begin{equation}\label{Lmoment}
\mathbb{E}\Big(\sup_{0\leq t\leq T}\|\int_{0}^{t}G(u(s))dL^{\alpha}(s)\|_{\theta}\Big)^{p}\leq C\mathbb{E}\Big(\int_{0}^{T}\|G(u(s))\|_{L_{HS}}^{\alpha}ds\Big)^{\frac{p}{\alpha}},
\end{equation}
which is a crucial moment inequality in investigating system \eqref{GinzburgLandau}. Note that it does not apply to the stochastic convolution in the Definition \eqref{mild},
as the integrand there depends on time.
\end{remark}

\section{Framework and main result}

Focusing on investigating  the local mild solution $u$  such that it is small
of order $\mathcal{O}(\varepsilon)$, we introduce the slow time scaling $T =\varepsilon^{2}t$. Let us split it into
\begin{equation*}
u(t)=\varepsilon a(\varepsilon^{2}t)+\varepsilon b(\varepsilon^{2}t),
\end{equation*}
with $a\in\mathcal{N}$ and $b\in\mathcal{S}$. By projecting and rescaling to the slow time scale, we obtain
\begin{equation}\label{a}
da(T)=[\mathcal{L}_{c}a(T)+\mathcal{F}_{c}(a(T)+b(T))]dT+\frac{1}{\varepsilon}G_{c}(\varepsilon a(T)+\varepsilon b(T))d\tilde{L}^{\alpha}(T)
\end{equation}
and
\begin{equation}\label{b}
db(T)=[\frac{1}{\varepsilon^{2}}\mathcal{A}_{s}b(T)+\mathcal{L}_{s}b(T)+\mathcal{F}_{s}(a(T)+b(T))]dT+\frac{1}{\varepsilon}G_{s}(\varepsilon a(T)+\varepsilon b(T))d\tilde{L}^{\alpha}(T),
\end{equation}
where $\tilde{L}^{\alpha}(T):=\varepsilon^{\frac{2}{\alpha}}L^{\alpha}(\varepsilon^{-2}T)$ is a rescaled version of the $\alpha$-stable L\'evy process. It is based on the fact that $\alpha$-stable L\'evy process $L^{\alpha}(t)$ is self-similar with Hurst index $1/\alpha$, i.e.,
\[
L^{\alpha}(ct)\stackrel{d}{=}  c^{1/\alpha} L^{\alpha}(t),~~~~c>0,
\]
where $``\stackrel{d}{=}"$ denotes equivalence (coincidence) in distribution.

Using the mild formulation, we rewrite the equations \eqref{a} and \eqref{b} into the integral form:
\begin{equation}\label{ia}
a(T)=a(0)+\int_{0}^{T}\mathcal{L}_{c}a(\tau)d\tau+\int_{0}^{T}\mathcal{F}_{c}(a(\tau)+b(\tau))d\tau+\frac{1}{\varepsilon}\int_{0}^{T}G_{c}(\varepsilon a(\tau)+\varepsilon b(\tau))d\tilde{L}^{\alpha}(\tau)
\end{equation}
and
\begin{align}\label{ib}
b(T)&=e^{\varepsilon^{-2}T\mathcal{A}_{s}}b(0)+\int_{0}^{T}e^{\varepsilon^{-2}(T-\tau)\mathcal{A}_{s}}\mathcal{L}_{s}b(\tau)d\tau+\int_{0}^{T}e^{\varepsilon^{-2}(T-\tau)\mathcal{A}_{s}}\mathcal{F}_{s}(a(\tau)+b(\tau))d\tau\\ \nonumber
    &~~~~+\frac{1}{\varepsilon}\int_{0}^{T}e^{\varepsilon^{-2}(T-\tau)\mathcal{A}_{s}}G_{s}(\varepsilon a(\tau)+\varepsilon b(\tau))d\tilde{L}^{\alpha}(\tau).
\end{align}
Denote the corresponding four terms arising in the right-hand side of system \eqref{ib} by $Q(T)$, $I(T)$, $J(T)$ and $K(T)$, respectively. That is
\begin{equation}\label{QIJK}
b(T)=Q(T)+I(T)+J(T)+K(T).
\end{equation}

We shall see later that $b$ is bounded and it's integral is small as long as $a$ is of order one (see Remark \ref{Q}, Lemma \ref{IJ}-\ref{K} and Lemma \ref{B} for the precise statements). Only $Q(T)$ with the initial condition and the term $K(T)$ are not $\mathcal{O}(\varepsilon)$, but $Q(T)$ is only of order one for very small times and the integral of $K(T)$ is $\mathcal{O}(\varepsilon)$. Thus by neglecting all $b$-dependent terms in \eqref{a} or \eqref{ia} and
expanding the $G_{c}$ term we obtain the amplitude equation
\begin{equation}\label{v}
d\varphi(T)=\mathcal{L}_{c}\varphi(T)dT+\mathcal{F}_{c}(\varphi(T))dT+[G'_{c}(0)\cdot \varphi(T)]d\tilde{L}^{\alpha}(T),~~~~\varphi(0)=a(0).
\end{equation}
Integrating \eqref{v} we obtain
\begin{equation}\label{iv}
\varphi(T)=a(0)+\int_{0}^{T}\mathcal{L}_{c}\varphi(\tau)d\tau+\int_{0}^{T}\mathcal{F}_{c}(\varphi(\tau))d\tau+\int_{0}^{T}[G'_{c}(0)\cdot \varphi(\tau)]d\tilde{L}^{\alpha}(\tau).
\end{equation}
Note that the noise in the SDE is still infinite dimensional,
but $G'_{c}(0)\cdot \varphi = \mathcal{P}_c[G'(0)\cdot \varphi]$ is a Hilbert-Schmidt operator that maps $\tilde{L}^{\alpha}(\tau)$ into the finite dimensional space $\mathcal{N}$.

Define the overall error between $u$ and $\varepsilon\varphi$ by $\mathcal{R}(\varepsilon^{2}t):=u(t)-\varepsilon \varphi(\varepsilon^{2}t)$ or
\begin{align}\nonumber
\mathcal{R}(T)&:=u(\varepsilon^{-2}T)-\varepsilon \varphi(T) \\ \nonumber
              &:=\varepsilon[a(T)-\varphi(T)+b(T)]\\ \label{huaR}
              &:=\varepsilon[a(T)-\varphi(T)+Q(T)+I(T)+J(T)+K(T)].
\end{align}

With our main assumptions we have the following main result on the approximation by
amplitude equation, which is proved later at the end of this section .

\begin{theorem}\label{MR}
Assume that $(\textbf{A1})$-$(\textbf{A5})$ hold. Let $u$ be the  mild solution of \eqref{GinzburgLandau}
with initial condition
\begin{equation*}
u(0)=\varepsilon a(0)+\varepsilon b(0),
\end{equation*}
where $a(0)\in\mathcal{N}$ and $b(0)\in\mathcal{S}$. The solution $\varphi$ of the amplitude equation \eqref{v} satisfies the initial condition $\varphi(0)=a(0)$. Then for any $p\in(0,\alpha)$, $T_{0}>0$ and all small $\kappa\in(0,\frac{2}{19})$,
provided
$\| u(0)\|_{\theta}\leq\varepsilon^{1-\kappa/2}$
we
obtain for the error  $\mathcal{R}$ defined in \eqref{huaR}
 that
\begin{equation*}
\mathbb{P}\Big(\|\mathcal{R}_c\|_{L^\infty([0,T_0];\mathcal{H}^{\theta})}
\geq\varepsilon^{2-19\kappa}\Big)\underset{\varepsilon\rightarrow0}{\longrightarrow}0\quad\text{and}\quad\mathbb{P}\Big(\|\mathcal{R}_s\|_{L^p([0,T_0];\mathcal{H}^{\theta})}\geq\varepsilon^{3-7\kappa }\Big)\underset{\varepsilon\rightarrow0}{\longrightarrow}0.
\end{equation*}
\end{theorem}
Let us first remark that in the previous result for the time $\tau_{ex}$ of existence
of our local mild solution we also have $\tau_{ex}>\varepsilon^{-2}T_{0}$ with high probability.

Now we need to introduce a stopping time in connection with process $(a,b)$. This stopping
time is equivalent to a cut-off in \eqref{GinzburgLandau} at order slightly bigger than $\varepsilon$. Also this stopping time is the reason, why we only need local solutions for the SPDE that might not exist for all times.

\begin{definition}\label{txin}
For the $\mathcal{N}\times\mathcal{S}$-valued stochastic process $(a,b)$ satisfying the integral equations \eqref{ia} and \eqref{ib}
, some time $T_{0}>0$ and small exponent $\kappa\in(0,\frac{2}{19})$, we define the stopping time $\tau^{*}$ as
\begin{equation}\label{tao}
\tau^{*}:=T_{0}\wedge\inf\{T>0:~\|a(T)\|_{\theta}>\varepsilon^{-\kappa}~~~or~~~\|b(T)\|_{\theta}>\varepsilon^{-2\kappa}\}.
\end{equation}
\end{definition}

\begin{remark}\label{Q}
In the decomposition of $b(T)$, for $Q(T)=e^{\varepsilon^{-2}T\mathcal{A}_{s}}b(0)$ the following estimates hold:
\begin{gather*}
\mathbb{E}\sup_{0\leq T\leq\tau^{*}}\|Q(T)\|_{\theta}^{p}\leq\mathbb{E}\sup_{0\leq T\leq\tau^{*}} Ce^{-\varepsilon^{-2}\rho pT}\|b(0)\|_{\theta}^{p}\leq C\|b(0)\|_{\theta}^{p},\\
\mathbb{E}\int_{0}^{\tau^{*}}\|Q(\tau)\|_{\theta}^{p}d\tau\leq\mathbb{E}C\int_{0}^{\tau^{*}}e^{-\varepsilon^{-2}\rho p\tau}\|b(0)\|_{\theta}^{p}d\tau\leq C\varepsilon^{2}\|b(0)\|_{\theta}^{p}.
\end{gather*}
\end{remark}

Before proving the main result Theorem \ref{MR}, we need to state some technical lemmas used later in the proof.
\begin{lemma}\label{IJ}
Assume that the assumptions $(\textbf{A1})$-$(\textbf{A5})$ hold. For $p\in(0,\alpha)$ and $\tau^{*}$ from Definition $\rm\ref{txin}$,
there exists a constant $C>0$ such that
\begin{equation}\label{I}
\mathbb{E}\sup_{0\leq T\leq\tau^{*}}\|I(T)\|_{\theta}^{p}\leq C\varepsilon^{2p-2\kappa p}
\end{equation}
and
\begin{equation}\label{J}
\mathbb{E}\sup_{0\leq T\leq\tau^{*}}\|J(T)\|_{\theta}^{p}\leq C\varepsilon^{2p-6\kappa p}.
\end{equation}
\end{lemma}
\begin{proof}
By virtue of the estimate \eqref{es},
\begin{align*}
\mathbb{E}\sup_{0\leq T\leq\tau^{*}}\|I(T)\|_{\theta}^{p}&=\mathbb{E}\sup_{0\leq T\leq\tau^{*}}\|\int_{0}^{T}e^{\varepsilon^{-2}(T-\tau)\mathcal{A}_{s}}\mathcal{L}_{s}b(\tau)d\tau\|_{\theta}^{p}\\
&\leq\mathbb{E}\sup_{0\leq T\leq\tau^{*}}[\int_{0}^{T}\|e^{\varepsilon^{-2}(T-\tau)\mathcal{A}_{s}}\mathcal{L}_{s}b(\tau)\|_{\theta}d\tau]^{p}\\
 &\leq C\varepsilon^{2\sigma p}\mathbb{E}\sup_{0\leq T\leq\tau^{*}}[\int_{0}^{T}e^{-\varepsilon^{-2}\rho(T-\tau)}(T-\tau)^{-\sigma}\|\mathcal{L}_{s}b(\tau)\|_{\theta-\sigma}d\tau]^{p}\\
&\leq C\varepsilon^{2\sigma p}\mathbb{E}\sup_{0\leq T\leq\tau^{*}}[\int_{0}^{T}e^{-\varepsilon^{-2}\rho(T-\tau)}(T-\tau)^{-\sigma}\|b(\tau)\|_{\theta}d\tau]^{p}\\
 &\leq C\varepsilon^{2\sigma p}\sup_{0\leq T\leq\tau^{*}}[\int_{0}^{T}e^{-\varepsilon^{-2}\rho(T-\tau)}(T-\tau)^{-\sigma}\varepsilon^{-2\kappa}d\tau]^{p}\\
&\leq C\varepsilon^{2p-2\kappa p}\sup_{0\leq T\leq\tau^{*}}[\int_{0}^{\varepsilon^{-2}\rho T}e^{-r}r^{-\sigma}dr]^{p}\\
&\leq C\varepsilon^{2p-2\kappa p}.
\end{align*}
Now, let's prove inequality \eqref{J},
\allowdisplaybreaks
\begin{eqnarray*}
\mathbb{E}\sup_{0\leq T\leq\tau^{*}}\|J(T)\|_{\theta}^{p}
&=&\mathbb{E}\sup_{0\leq T\leq\tau^{*}}\|\int_{0}^{T}e^{\varepsilon^{-2}(T-\tau)\mathcal{A}_{s}}\mathcal{F}_{s}(a(\tau)+b(\tau))d\tau\|_{\theta}^{p}\\
&\leq&\mathbb{E}\sup_{0\leq T\leq\tau^{*}}[\int_{0}^{T}\|e^{\varepsilon^{-2}(T-\tau)\mathcal{A}_{s}}\mathcal{F}_{s}(a(\tau)+b(\tau))\|_{\theta}d\tau]^{p}\\
&\leq& C\varepsilon^{2\sigma p}\mathbb{E}\sup_{0\leq T\leq\tau^{*}}[\int_{0}^{T}e^{-\varepsilon^{-2}\rho(T-\tau)}(T-\tau)^{-\sigma}\|\mathcal{F}_{s}(a(\tau)+b(\tau))\|_{\theta-\sigma}d\tau]^{p}\\
 &\leq& C\varepsilon^{2\sigma p}\mathbb{E}\sup_{0\leq T\leq\tau^{*}}[\int_{0}^{T}e^{-\varepsilon^{-2}\rho(T-\tau)}(T-\tau)^{-\sigma}\|a(\tau)+b(\tau)\|_{\theta}^{3}d\tau]^{p}\\
 &\leq& C\varepsilon^{2\sigma p}\mathbb{E}\sup_{0\leq T\leq\tau^{*}}[\int_{0}^{T}e^{-\varepsilon^{-2}\rho(T-\tau)}(T-\tau)^{-\sigma}\varepsilon^{-6\kappa}d\tau]^{p}\\
&\leq& C\varepsilon^{2p-6\kappa p}\mathbb{E}\sup_{0\leq T\leq\tau^{*}}[\int_{0}^{\varepsilon^{-2}\rho T}e^{-r}r^{-\sigma}dr]^{p}\\
 &\leq& C\varepsilon^{2p-6\kappa p}.
\end{eqnarray*}
This finishes the proof.
\end{proof}

For $I$ and $J$, we have uniform bounds in time that show the smallness of both terms. The bounds of the error term in Lemma \ref{ek} will thus be determined via the estimate on $K$.
But here we encounter serious problems. By large jumps due to the noise,
we are no longer able to show that $K$ is small uniformly in time.
We can only verify bounds for $K$ in $L^{p}([0,\tau^{*}];\mathcal{H}^{\theta})$.
\begin{lemma}\label{K}
Assume the setting of Lemma \ref{IJ}. Then it holds for every $p\in(0,\alpha)$ that
\begin{equation}\label{Kt}
\mathbb{E}\int_{0}^{\tau^{*}}\|K(\tau)\|_{\theta}^{p}d\tau\leq C\varepsilon^{(\frac{2}{\alpha}-2\kappa)p+2}.
\end{equation}
\end{lemma}

\begin{proof}
To show the estimate \eqref{Kt}, we use the Riesz-Nagy-trick \cite{RN} here, which embeds a contraction semigroup into a larger Hilbert-space,
where it is a group defined for all times.
Let $\lambda_{0}$ be a positive constant less than $\lambda_{n+1}$ but close to it. For any $p\in(0,\alpha)$, by using the maximal inequality for stochastic convolutions \cite{BH09} based on the Riesz-Nagy theorem (as $\mathcal{A}_{s}+\lambda_{0}\mathcal{I}$ generates a contraction semigroup on $\mathcal{S}$), the condition \eqref{g} for $G$, and the definition of the stopping time $\tau^{*}$,\allowdisplaybreaks
\begin{eqnarray*}
\mathbb{E}\int_{0}^{\tau^{*}}\|K(\tau)\|_{\theta}^{p}d\tau
&=&\mathbb{E}\int_{0}^{\tau^{*}}\|\frac{1}{\varepsilon}\int_{0}^{\tau}e^{\varepsilon^{-2}(\tau-r)\mathcal{A}_{s}}G_{s}(\varepsilon a(r)+\varepsilon b(r))d\tilde{L}^{\alpha}(r)\|_{\theta}^{p}d\tau
\\
&\leq&\mathbb{E}\int_{0}^{T_{0}}\|\mathds{1}_{[0, \tau^{*}]}\frac{1}{\varepsilon}\int_{0}^{\tau}e^{\varepsilon^{-2}(\tau-r)\mathcal{A}_{s}}G_{s}(\varepsilon a(r)+\varepsilon b(r))d\tilde{L}^{\alpha}(r)\|_{\theta}^{p}d\tau
\\
&\leq&\mathbb{E}\int_{0}^{T_{0}}\|\frac{1}{\varepsilon}\int_{0}^{\tau\wedge\tau^{*}}e^{\varepsilon^{-2}(\tau-r)\mathcal{A}_{s}}G_{s}(\varepsilon a(r)+\varepsilon b(r))d\tilde{L}^{\alpha}(r)\|_{\theta}^{p}d\tau
\\
&=&\mathbb{E}\int_{0}^{T_{0}}\|\frac{1}{\varepsilon}\int_{0}^{\tau}\mathds{1}_{[0, \tau^{*}]}e^{\varepsilon^{-2}(\tau-r)\mathcal{A}_{s}}G_{s}(\varepsilon a(r)+\varepsilon b(r))d\tilde{L}^{\alpha}(r)\|_{\theta}^{p}d\tau
\\
&=&\frac{1}{\varepsilon^{p}}\int_{0}^{T_{0}}e^{-\varepsilon^{-2}\lambda_{0}p\tau}\mathbb{E}\|\int_{0}^{\tau}\mathds{1}_{[0, \tau^{*}]}e^{\varepsilon^{-2}(\tau-r)(\mathcal{A}_{s}+\lambda_{0}\mathcal{I})}e^{\varepsilon^{-2}\lambda_{0}r}G_{s}(\varepsilon a(r)+\varepsilon b(r))d\tilde{L}^{\alpha}(r)\|_{\theta}^{p}d\tau\\
                                                                             &\leq& C\int_{0}^{T_{0}}e^{-\varepsilon^{-2}\lambda_{0}p\tau}\mathbb{E}\Big[\int_{0}^{\tau}e^{\alpha\varepsilon^{-2}\lambda_{0}r}\mathds{1}_{[0, \tau^{*}]}\| a(r)+ b(r)\|_{\theta}^{\alpha}dr\Big]^{\frac{p}{\alpha}}d\tau
 \\
 &\leq& C\int_{0}^{T_{0}}e^{-\varepsilon^{-2}\lambda_{0}p\tau}\Big[\int_{0}^{\tau}\varepsilon^{-2\alpha\kappa}e^{\alpha\varepsilon^{-2}\lambda_{0}r}dr\Big]^{\frac{p}{\alpha}}d\tau
\\
&\leq& C\varepsilon^{\frac{2p}{\alpha}-2\kappa p}\int_{0}^{T_{0}}e^{-\varepsilon^{-2}\lambda_{0}p\tau}\Big[e^{\alpha\varepsilon^{-2}\lambda_{0}\tau}-1\Big]^{\frac{p}{\alpha}}d\tau
\\
&\leq& C\varepsilon^{(\frac{2}{\alpha}-2\kappa)p+2}.
\end{eqnarray*}
This confirms that the estimate \eqref{Kt} holds.
\end{proof}
\begin{corollary}\label{Co}
The expression \eqref{QIJK} for $b$, H\"{o}lder's inequality, triangle inequality, equivalence of norms and Lemmas \ref{IJ} and \ref{K} provide
\begin{align*}
  \mathbb{E}\Big[\int_{0}^{T_{0}}\mathds{1}_{[0, \tau^{*}]}(\tau)\|b(\tau)\|_{\theta}^{p}d\tau\Big]&\leq C\mathbb{E}\Big[\int_{0}^{T_{0}}\mathds{1}_{[0, \tau^{*}]}(\tau)\|b(\tau)\|_{\theta}^{\alpha}d\tau\Big]^{\frac{p}{\alpha}}\\
  &\leq C_\alpha\mathbb{E}\Big[\int_{0}^{T_{0}}\mathds{1}_{[0, \tau^{*}]}(\tau)(\|Q(\tau)\|_{\theta}^{\alpha}+\|I(\tau)\|_{\theta}^{\alpha}+\|J(\tau)\|_{\theta}^{\alpha}+\|K(\tau)\|_{\theta}^{\alpha})d\tau\Big]^{\frac{p}{\alpha}}\\
  &\leq C_\alpha\mathbb{E}\Big[\varepsilon^{2p}\|b(0)\|_{\theta}^{p}+\sup_{0\leq\tau\leq\tau^{*}}\|I(\tau)\|_{\theta}^{p}+\sup_{0\leq\tau\leq\tau^{*}}\|J(\tau)\|_{\theta}^{p}+\sup_{0\leq \tau\leq\tau^{*}}\int_{0}^{\tau}\|K(r)\|_{\theta}^{p}dr\Big]\\
  &\leq C_{\alpha,p}\varepsilon^{2p-6\kappa p}.
\end{align*}
\end{corollary}

Note that this $L^p$-bound on $b$ is not sufficient to obtain the estimate and remove the stopping time.
Here we will need uniform bounds both on $a$ and $b$. This will be done in the following results.

Let us rewrite the equation \eqref{ia} for $a$ as the amplitude equation plus an error term (or residual):
\begin{equation}\label{ra}
a(T)=a(0)+\int_{0}^{T}[\mathcal{L}_{c}a(\tau)+\mathcal{F}_{c}(a(\tau))]d\tau+\int_{0}^{T}G_{c}'(0)\cdot a(\tau)d\tilde{L}^{\alpha}(\tau)+R(T),
\end{equation}
where the error term is given by
\begin{align}\nonumber
R(T)&=\int_{0}^{T}[3\mathcal{F}_{c}(a(\tau),a(\tau),b(\tau))+3\mathcal{F}_{c}(a(\tau),b(\tau),b(\tau))+\mathcal{F}_{c}(b(\tau))]d\tau\\ \label{R}
    &\,\,\,\,~~~+\int_{0}^{T}\Big[\frac{1}{\varepsilon}G_{c}(\varepsilon a(\tau)+\varepsilon b(\tau))-G_{c}'(0)\cdot a(\tau)\Big]d\tilde{L}^{\alpha}(\tau).
\end{align}

\begin{lemma}\label{ek}
For any $p\in(0,\alpha)$, there exists a constant $C_p>0$ such that
\begin{equation*}
\mathbb{E}\sup_{0\leq T\leq\tau^{*}}\|\int_{0}^{T}\mathcal{F}_{c}(a(\tau),a(\tau),b(\tau))d\tau\|_{\theta}^{p}\leq C_p(\varepsilon^{\frac{2}{\alpha}p-8\kappa p}+\varepsilon^{2p-2\kappa p}\|b(0)\|_{\theta}^{p})
\end{equation*}
\end{lemma}
\begin{proof}
Using $b=Q+I+J+K$ from \eqref{QIJK}, by brute force expansion of the cubic,
\begin{align}\nonumber
  \int_{0}^{T}\mathcal{F}_{c}(a(\tau),a(\tau),b(\tau))d\tau&=\int_{0}^{T}\mathcal{F}_{c}(a(\tau),a(\tau),Q(\tau))d\tau+\int_{0}^{T}\mathcal{F}_{c}(a(\tau),a(\tau),I(\tau))d\tau\\ \nonumber
                                                          &\,\,\,\,\,\,+\int_{0}^{T}\mathcal{F}_{c}(a(\tau),a(\tau),J(\tau))d\tau+\int_{0}^{T}\mathcal{F}_{c}(a(\tau),a(\tau),K(\tau))d\tau\\ \label{R4}
                                                          &:=R_{1,1}(T)+R_{1,2}(T)+R_{1,3}(T)+R_{1,4}(T).
\end{align}
Now we estimate each term separately. Since all $\mathcal{H}^{\theta}$-norms are equivalent on $\mathcal{N}$, by Definition \ref{txin},
\begin{align*}
\mathbb{E}\sup_{0\leq T\leq\tau^{*}}\|R_{1,1}(T)\|_{\theta}^{p}
&\leq
C\mathbb{E}\sup_{0\leq T\leq\tau^{*}}\|R_{1,1}(T)\|_{\theta-\sigma}^{p}
\leq
C\mathbb{E}\sup_{0\leq T\leq\tau^{*}}\Big[\int_{0}^{T}\|\mathcal{F}_{c}(a(\tau),a(\tau),Q(\tau))\|_{\theta-\sigma}d\tau\Big]^{p}
\\
&\leq
C\mathbb{E}\sup_{0\leq T\leq\tau^{*}}\Big[\int_{0}^{T}\|a(\tau)\|_{\theta}^{2}\|Q(\tau)\|_{\theta}d\tau\Big]^{p}
\leq C\varepsilon^{-2\kappa p}\Big[\int_{0}^{T_{0}}\|e^{\varepsilon^{-2}\tau\mathcal{A}_{s}}b(0)\|_{\theta}d\tau\Big]^{p}
\\
&\leq
C\varepsilon^{2p-2\kappa p}\|b(0)\|_{\theta}^{p}.
\end{align*}
For $R_{1,2}(T)$, by the estimate \eqref{I}, we argue that
\begin{align*}
\mathbb{E}\sup_{0\leq T\leq\tau^{*}}\|R_{1,2}(T)\|_{\theta}^{p}
&\leq
C\mathbb{E}\sup_{0\leq T\leq\tau^{*}}\|R_{1,2}(T)\|_{\theta-\sigma}^{p}
\leq
C\mathbb{E}\sup_{0\leq T\leq\tau^{*}}\Big[\int_{0}^{T}\|\mathcal{F}_{c}(a(\tau),a(\tau),I(\tau))\|_{\theta-\sigma}d\tau\Big]^{p}
\\
&\leq
C\mathbb{E}\sup_{0\leq T\leq\tau^{*}}\Big[\int_{0}^{T}\|a(\tau)\|_{\theta}^{2}\|I(\tau)\|_{\theta}d\tau\Big]^{p}
\leq
C\varepsilon^{-2\kappa p}\mathbb{E}\sup_{0\leq T\leq\tau^{*}}\int_{0}^{T}\|I(\tau)\|_{\theta}^{p}d\tau
\\
&\leq C\varepsilon^{2p-4\kappa p}.
\end{align*}
Due to the estimate \eqref{J}, we get
\begin{align*}
\mathbb{E}\sup_{0\leq T\leq\tau^{*}}\|R_{1,3}(T)\|_{\theta}^{p}
&\leq
C\mathbb{E}\sup_{0\leq T\leq\tau^{*}}\|R_{1,3}(T)\|_{\theta-\sigma}^{p}
\leq C\mathbb{E}\sup_{0\leq T\leq\tau^{*}}\Big[\int_{0}^{T}\|\mathcal{F}_{c}(a(\tau),a(\tau),J(\tau))\|_{\theta-\sigma}d\tau\Big]^{p}
\\
&\leq C\mathbb{E}\sup_{0\leq T\leq\tau^{*}}\Big[\int_{0}^{T}\|a(\tau)\|_{\theta}^{2}\|J(\tau)\|_{\theta}d\tau\Big]^{p}
\leq
C\varepsilon^{-2\kappa p}\mathbb{E}\sup_{0\leq T\leq\tau^{*}}\int_{0}^{T}\|J(\tau)\|_{\theta}^{p}d\tau\\
&\leq C\varepsilon^{2p-8\kappa p}.
\end{align*}
Based on the estimate \eqref{Kt}, we obtain
\begin{align*}
\mathbb{E}\sup_{0\leq T\leq\tau^{*}}\|R_{1,4}(T)\|_{\theta}^{p}&\leq C\mathbb{E}\sup_{0\leq T\leq\tau^{*}}\|R_{1,4}(T)\|_{\theta-\sigma}^{p}
\leq
C\mathbb{E}\sup_{0\leq T\leq\tau^{*}}\Big[\int_{0}^{T}\|\mathcal{F}_{c}(a(\tau),a(\tau),K(\tau))\|_{\theta-\sigma}d\tau\Big]^{p}
\\&\leq
C\mathbb{E}\sup_{0\leq T\leq\tau^{*}}\Big[\int_{0}^{T}\|a(\tau)\|_{\theta}^{2}\|K(\tau)\|_{\theta}d\tau\Big]^{p}
\leq
C\varepsilon^{-2\kappa p}\mathbb{E}\sup_{0\leq T\leq\tau^{*}}\int_{0}^{T}\|K(\tau)\|_{\theta}^{p}d\tau
\\
&\leq
C\varepsilon^{(\frac{2}{\alpha}-4\kappa)p+2}.
\end{align*}
Thus we finished the proof.
\end{proof}
The following two lemmas can be treated in a similar manner.
We omit the proofs.
\begin{lemma}
Assume the setting of Lemma \ref{ek}. For any $p\in(0,\alpha)$, there exists a constant $C_p>0$ such that
\begin{equation*}
\mathbb{E}\sup_{0\leq T\leq\tau^{*}}\|\int_{0}^{T}\mathcal{F}_{c}(a(\tau),b(\tau),b(\tau))d\tau\|_{\theta}^{p}\leq C_p(\varepsilon^{2p-13\kappa p}+\varepsilon^{2p-\kappa p}\|b(0)\|_{\theta}^{2p}).
\end{equation*}
\end{lemma}

\begin{lemma}
Assume the setting of Lemma \ref{ek}. For any $p\in(0,\alpha)$, there exists a constant $C_p>0$ such that
\begin{equation*}
\mathbb{E}\sup_{0\leq T\leq\tau^{*}}\|\int_{0}^{T}\mathcal{F}_{c}(b(\tau))d\tau\|_{\theta}^{p}\leq C_p(\varepsilon^{\frac{4}{\alpha}p-18\kappa p}+\varepsilon^{2p}\|b(0)\|_{\theta}^{3p}).
\end{equation*}
\end{lemma}

\begin{lemma}\label{EG}
Assume the setting of Lemma \ref{ek}. For any $p\in(0,\alpha)$, there exists a constant $C_p>0$ such that
\begin{equation*}
\mathbb{E}\sup_{0\leq T\leq\tau^{*}}\|\int_{0}^{T}\Big[\frac{1}{\varepsilon}G_{c}(\varepsilon a(\tau)+\varepsilon b(\tau))-G_{c}'(0)\cdot a(\tau)\Big]d\tilde{L}^{\alpha}(\tau)\|_{\theta}^{p}\leq C_p\varepsilon^{p-6\kappa p}.
\end{equation*}
\end{lemma}
\begin{proof}
By the moment inequality \eqref{Lmoment}, a direct calculation yields
\begin{align}\nonumber
   \mathbb{E}\sup_{0\leq T\leq\tau^{*}}\|\int_{0}^{T}&\Big[\frac{1}{\varepsilon}G_{c}(\varepsilon a(\tau)+\varepsilon b(\tau))-G_{c}'(0)\cdot a(\tau)\Big]d\tilde{L}^{\alpha}(\tau)\|_{\theta}^{p}   \\ \nonumber
  &\leq C\mathbb{E}\Big[\int_{0}^{T_{0}}\mathds{1}_{[0, \tau^{*}]}(\tau)\|\frac{1}{\varepsilon}G_{c}(\varepsilon a(\tau)+\varepsilon b(\tau))-G_{c}'(0)\cdot a(\tau)\|_{L_{HS}}^{\alpha}d\tau\Big]^{\frac{p}{\alpha}}\\ \label{GG}
  &\leq C\mathbb{E}\Big[\int_{0}^{T_{0}}\mathds{1}_{[0, \tau^{*}]}(\tau)\|\frac{1}{\varepsilon}G(\varepsilon a(\tau)+\varepsilon b(\tau))-G'(0)\cdot a(\tau)\|_{L_{HS}}^{\alpha}d\tau\Big]^{\frac{p}{\alpha}}.
\end{align}
Utilizing the Taylor formula and $G(0)=0$, we can check
\begin{align*}
  \lefteqn{\frac{1}{\varepsilon}G(\varepsilon a(\tau)+\varepsilon b(\tau))-G'(0)\cdot a(\tau)}\\
  &=\frac{1}{\varepsilon}\Big[G(0)+G'(0)(\varepsilon a(\tau)+\varepsilon b(\tau))
  +\frac{1}{2}G''(z(\tau))\cdot(\varepsilon a(\tau)+\varepsilon b(\tau), \varepsilon a(\tau)+\varepsilon b(\tau))\Big]-G'(0)\cdot a(\tau)\\
  &=G'(0)\cdot b(\tau)+\frac{\varepsilon}{2}G''(z(\tau))\cdot (a(\tau)+b(\tau), a(\tau)+b(\tau)),
\end{align*}
where $z(\tau)$ is a vector on the line segment connecting $0$ and $\varepsilon a(\tau)+\varepsilon b(\tau)$. With the conditions \eqref{gv} and \eqref{gvw}, we  obtain
\begin{align*}
   \|\frac{1}{\varepsilon}G(\varepsilon a(\tau)+\varepsilon b(\tau))-G'(0)\cdot a(\tau)\|_{L_{HS}}^{\alpha}
   &=\|G'(0)\cdot b(\tau)+\frac{\varepsilon}{2}G''(z(\tau))\cdot (a(\tau)+b(\tau), a(\tau)+b(\tau))\|_{L_{HS}}^{\alpha}\\
   &\leq C(\|b(\tau)\|_{\theta}^{\alpha}+\varepsilon^{\alpha}\|a(\tau)\|_{\theta}^{2\alpha}+\varepsilon^{\alpha}\|b(\tau)\|_{\theta}^{2\alpha}).
\end{align*}
After substituting the above estimate back into \eqref{GG}, we derive
\begin{align*}
  &\mathbb{E}\sup_{0\leq T\leq\tau^{*}}\|\int_{0}^{T}\Big[\frac{1}{\varepsilon}G_{c}(\varepsilon a(\tau)+\varepsilon b(\tau))-G_{c}'(0)\cdot a(\tau)\Big]d\tilde{L}^{\alpha}(\tau)\|_{\theta}^{p}\\
  &\leq C\mathbb{E}\Big[\int_{0}^{T_{0}}\mathds{1}_{[0, \tau^{*}]}(\tau)(\|b(\tau)\|_{\theta}^{\alpha}+\varepsilon^{\alpha}\|a(\tau)\|_{\theta}^{2\alpha}+\varepsilon^{\alpha}\|b(\tau)\|_{\theta}^{2\alpha})d\tau\Big]^{\frac{p}{\alpha}}\\
  &\leq C_p\varepsilon^{p-4\kappa p}+C_p\mathbb{E}\Big[\int_{0}^{T_{0}}\mathds{1}_{[0, \tau^{*}]}(\tau)\|b(\tau)\|_{\theta}^{\alpha}d\tau\Big]^{\frac{p}{\alpha}},
\end{align*}
where the last estimate is obtained via the definition of $ \tau^{*}$.
Hence, using Corollary \ref{Co} leads to
\begin{equation*}
\mathbb{E}\sup_{0\leq T\leq\tau^{*}}\|\int_{0}^{T}\Big[\frac{1}{\varepsilon}G_{c}(\varepsilon a(\tau)+\varepsilon b(\tau))-G_{c}'(0)\cdot a(\tau)\Big]d\tilde{L}^{\alpha}(\tau)\|_{\theta}^{p}\leq C_p\varepsilon^{p-6\kappa p}.
\end{equation*}
\end{proof}

It should be pointed out here that Lemma \ref{ek}-\ref{EG} reveal that the remainder $R$ defined in \eqref{R} satisfies the following estimate.

\begin{lemma}
In addition to the assumptions $(\textbf{A1})$-$(\textbf{A5})$, the suitable condition $\|b(0)\|_{\theta}\leq\varepsilon^{-\kappa}$ is set. Then for any $p\in(0,\alpha)$, there exists a constant $C_p>0$ such that
\begin{equation}\label{Re}
\mathbb{E}\sup_{0\leq T\leq\tau^{*}}\|R(T)\|_{\theta}^{p}\leq C_p\varepsilon^{p-18\kappa p}.
\end{equation}
\end{lemma}

As we have a good bound on the residual $R$,
our focus will now be on the solution of the amplitude equation \eqref{v} in conjunction with \eqref{a}.
The following uniform bound  depending on the initial condition
$a(0)$ for the solution $\varphi$ is necessary to bound $a$ and later to remove the stopping time from the error estimate. It is worthwhile to note that the following lemma would also allow us to confirm the existence of global solutions for the amplitude equation.

\begin{lemma}
Under assumptions $(\textbf{A1})$-$(\textbf{A5})$, for any $p\in(1,\alpha)$, there exists a constant $C_p>0$ such that
\begin{equation}\label{fine}
\mathbb{E}\sup_{0\leq T\leq T_{0}}\|\varphi(T)\|_{\theta}^{p}
\leq C_p\big(\|a(0)\|_{\theta}^{p}+1\big).
\end{equation}
\end{lemma}
\begin{proof}
Define a smooth function $f$ on $\mathcal{H}$ by
\begin{equation}\label{f}
f(\cdot)=\big(\|\cdot\|^{2}+1\big)^{\frac{p}{2}}.
\end{equation}
As a result, for any $x, h\in\mathcal{H}$,
\begin{equation}\label{ff}
f'(x)h=\frac{p}{\big(\|x\|^{2}+1\big)^{1-\frac{p}{2}}}\langle x,h\rangle
\end{equation}
and
\begin{equation}\label{fff}
f''(x)(h,h)=\frac{p}{\big(\|x\|^{2}+1\big)^{1-\frac{p}{2}}}\langle h,h\rangle+\frac{p(p-2)}{\big(\|x\|^{2}+1\big)^{2-\frac{p}{2}}}\langle x,h\rangle\langle x,h\rangle\leq\frac{p(p-1)}{\big(\|x\|^{2}+1\big)^{1-\frac{p}{2}}}\|h\|^{2}.
\end{equation}
Moreover,
\begin{equation*}
\|f'(x)\|\leq\frac{C_p\|x\|}{\big(\|x\|^{2}+1\big)^{1-\frac{p}{2}}}\leq C_p\|x\|^{p-1},\quad \|f''(x)\|\leq\frac{C_p}{\big(\|x\|^{2}+1\big)^{1-\frac{p}{2}}}\leq C_p.
\end{equation*}

It follows from It\^o's formula (\cite[Theorem 4.4.7]{App}) that
\begin{align}\nonumber
f(\varphi(T))&=f(a(0))+\int_{0}^{T}\Big[\frac{\langle\mathcal{L}_{c}\varphi(\tau),p\varphi(\tau)\rangle}{\big(\|\varphi(\tau)\|^{2}+1\big)^{1-\frac{p}{2}}}+\frac{\langle\mathcal{F}_{c}(\varphi(\tau)),p\varphi(\tau)\rangle}{\big(\|\varphi(\tau)\|^{2}+1\big)^{1-\frac{p}{2}}}\Big]d\tau
\\ \nonumber
&~~~+\sum_{k=1}^{\infty}\int_{0}^{T}\int_{|y|\geq1}[f(\varphi(\tau-)+G'_{c}(0)\cdot \varphi(\tau)ye_{k})-f(\varphi(\tau-))]N_{k}(d\tau,dy)
\\ \nonumber
&~~~+\sum_{k=1}^{\infty}\int_{0}^{T}\int_{|y|<1}[f(\varphi(\tau-)+G'_{c}(0)\cdot \varphi(\tau)ye_{k})-f(\varphi(\tau-))]\tilde{N}_{k}(d\tau,dy)\\ \nonumber
&~~~+\sum_{k=1}^{\infty}\int_{0}^{T}\int_{|y|<1}\Big[f(\varphi(\tau-)+G'_{c}(0)\cdot \varphi(\tau)ye_{k})-f(\varphi(\tau-))-\frac{\langle G'_{c}(0)\cdot \varphi(\tau)ye_{k},p\varphi(\tau-)\rangle}{\big(\|\varphi(\tau-)\|^{2}+1\big)^{1-\frac{p}{2}}}\Big]\nu_{\alpha}(dy)d\tau\\ \label{f4}
             &:=\big(\|a(0)\|^{2}+1\big)^{\frac{p}{2}}+I_{1}(T)+I_{2}(T)+I_{3}(T)+I_{4}(T).
\end{align}

Define a stopping time
$\mathcal{T}:=T_{0}\wedge\inf\{t>0: \|\varphi(t)\|>n\}\leq T_{0}$, $n\in\mathbb{N}$.
Note that we need to use a stopping time $\mathcal{T}$ in order to have   $\varphi$ bounded.
A-priori we do not know that the moments of $\varphi$ are finite, thus we consider the process only up to the stopping time $\mathcal{T}$
which ensures this.

Taking account of \eqref{ff}, assumption $(\textbf{A2})$ and the bound on $\mathcal{F}$ from \eqref{fc}, we obtain
\begin{equation}\label{I1}
\mathbb{E}\sup_{0\leq T\leq\mathcal{T}}\|I_{1}(T)\|\leq C_p\mathbb{E}\int_{0}^{\mathcal{T}}\|\varphi(\tau)\|^{p}d\tau.
\end{equation}
Using \eqref{ff} again, we obtain
\begin{align}\nonumber
\mathbb{E}\sup_{0\leq T\leq\mathcal{T}}\|I_{2}(T)\|
&\leq C\sum_{k=1}^{\infty}\mathbb{E}\Big(\int_{0}^{\mathcal{T}}\int_{|y|\geq1}|f(\varphi(\tau)+G'_{c}(0)\cdot \varphi(\tau)y e_{k})-f(\varphi(\tau))|N_{k}(d\tau,dy)\Big)
\\ \nonumber
&=C\sum_{k=1}^{\infty}\mathbb{E}\Big(\int_{0}^{\mathcal{T}}\int_{|y|\geq1}|f(\varphi(\tau)+G'_{c}(0)\cdot \varphi(\tau)ye_{k})-f(\varphi(\tau))|\nu_{\alpha}(dy)d\tau\Big)
\\ \nonumber
&\leq C\sum_{k=1}^{\infty}\mathbb{E}\Big(\int_{0}^{\mathcal{T}}\int_{|y|\geq1}\int_{0}^{1}\|f'(\varphi(\tau)+\xi G'_{c}(0)\cdot \varphi(\tau)ye_{k})\|d\xi\|G'_{c}(0)\cdot \varphi(\tau)ye_{k}\|\nu_{\alpha}(dy)d\tau\Big)
\\ \nonumber
&\leq C_p\sum_{k=1}^{\infty}\mathbb{E}\Big(\int_{0}^{\mathcal{T}}\int_{|y|\geq1}(\|\varphi(\tau)\|^{p-1}+\| G'_{c}(0)\cdot \varphi(\tau)ye_{k}\|^{p-1})\|G'_{c}(0)\cdot \varphi(\tau)ye_{k}\|\nu_{\alpha}(dy)d\tau\Big)
\\ \nonumber
 &\leq C_p\Big[\Big(\sum_{k=1}^{\infty}\beta_{k}\int_{|y|\geq1}|y|\nu_{\alpha}(dy)+\sum_{k=1}^{\infty}\beta_{k}^{p}\int_{|y|\geq1}|y|^{p}\nu_{\alpha}(dy)\Big)\mathbb{E}\int_{0}^{\mathcal{T}}\|\varphi(\tau)\|^{p}d\tau\Big]
 \\ \label{I2}
 &\leq C_p\mathbb{E}\int_{0}^{\mathcal{T}}\|\varphi(\tau)\|^{p}d\tau.
\end{align}

Since the compensated compound Poisson process $\tilde{N}$ is a martingale, by the Burkholder-Davis-Gundy inequality \cite{Chow}  and \eqref{ff}, we figure out
\begin{align}\nonumber
\mathbb{E}\sup_{0\leq T\leq\mathcal{T}}\|I_{3}(T)\|
&\leq C\sum_{k=1}^{\infty}\mathbb{E}\Big(\int_{0}^{\mathcal{T}}\int_{|y|<1}|f(\varphi(\tau)+G'_{c}(0)\cdot \varphi(\tau)ye_{k})-f(\varphi(\tau))|^{2}
\nu_{\alpha}(dy)d\tau\Big)^{\frac{1}{2}}
\\ \nonumber
&\leq C\sum_{k=1}^{\infty}\mathbb{E}\Big(\int_{0}^{\mathcal{T}}\int_{|y|<1}\int_{0}^{1}\|f'(\varphi(\tau)+\xi G'_{c}(0)\cdot \varphi(\tau)ye_{k})\|^{2}d\xi\|G'_{c}(0)\cdot \varphi(\tau)y e_{k}\|^{2}\nu_{\alpha}(dy)d\tau\Big)^{\frac{1}{2}}
\\ \nonumber
&\leq C_p\sum_{k=1}^{\infty}\mathbb{E}\Big(\int_{0}^{\mathcal{T}}\int_{|y|<1}(\|\varphi(\tau)\|^{2p-2}+\| G'_{c}(0)\cdot \varphi(\tau)ye_{k}\|^{2p-2})\|G'_{c}(0)\cdot \varphi(\tau)ye_{k}\|^{2}\nu_{\alpha}(dy)d\tau\Big)^{\frac{1}{2}}
\\ \nonumber
 &\leq C_p\Big(\sum_{k=1}^{\infty}\beta_{k}\Big[\int_{|y|<1}|y|^{2}\nu_{\alpha}(dy)\Big]^{\frac{1}{2}}+\sum_{k=1}^{\infty}\beta_{k}^{p}\Big[\int_{|y|<1}|y|^{2p}\nu_{\alpha}(dy)\Big]^{\frac{1}{2}}\Big)\mathbb{E}\Big[\int_{0}^{\mathcal{T}}\|\varphi(\tau)\|^{2p}d\tau\Big]^{\frac{1}{2}}
 \\ \nonumber
&\leq C_p\mathbb{E}\Big[\sup_{0\leq \tau\leq\mathcal{T}}\|\varphi(\tau)\|^{p}\int_{0}^{\mathcal{T}}\|\varphi(\tau)\|^{p}d\tau\Big]^{\frac{1}{2}}
\\ \label{I3}
 &\leq C_p\mathbb{E}\sup_{0\leq \tau\leq\mathcal{T}}\|\varphi(\tau)\|^{p}+C_p\mathbb{E}\int_{0}^{\mathcal{T}}\sup_{0\leq s\leq \tau}\|\varphi(s)\|^{p}d\tau,
\end{align}
where we used Young's inequality in the last inequality.
The Taylor's expansion and \eqref{ff}-\eqref{fff} imply that

\begin{align}\nonumber
\mathbb{E}\sup_{0\leq T\leq\mathcal{T}}\|I_{4}(T)\|
&\leq C\sum_{k=1}^{\infty}\mathbb{E}\int_{0}^{\mathcal{T}}\int_{|y|<1}\Big|f(\varphi(\tau)+G'_{c}(0)\cdot \varphi(\tau)ye_{k})-f(\varphi(\tau))-\frac{\langle G'_{c}(0)\cdot \varphi(\tau)ye_{k},p\varphi(\tau)\rangle}{\big(\|\varphi(\tau)\|^{2}+1\big)^{1-\frac{p}{2}}}\Big|\nu_{\alpha}(dy)d\tau\\ \nonumber
&\leq C_p\sum_{k=1}^{\infty}\mathbb{E}\int_{0}^{\mathcal{T}}\int_{|y|<1}\frac{\|G'_{c}(0)\cdot \varphi(\tau)ye_{k}\|^{2}}{\big(\|\varphi(\tau)\|^{2}+1\big)^{1-\frac{p}{2}}}\nu_{\alpha}(dy)d\tau\\ \nonumber
&\leq C_p\sum_{k=1}^{\infty}\beta_{k}^{2}\int_{|y|<1}|y|^{2}\nu_{\alpha}(dy)\mathbb{E}\int_{0}^{\mathcal{T}}\|\varphi(\tau)\|^{p}d\tau\\ \label{I4}
&\leq C_p\mathbb{E}\int_{0}^{\mathcal{T}}\|\varphi(\tau)\|^{p}d\tau.
\end{align}

Combining estimates \eqref{f4}-\eqref{I4} yields
\begin{equation*}
\mathbb{E}\sup_{0\leq t\leq\mathcal{T}}\|\varphi(t)\|^{p}\leq C_p\big(\|a(0)\|^{p}+1\big)+C_p\mathbb{E}\int_{0}^{\mathcal{T}}\sup_{0\leq s\leq t}\|\varphi(s)\|^{p}dt.
\end{equation*}
By an application of Gronwall's lemma, we thus derive
\begin{equation*}
\mathbb{E}\sup_{0\leq t\leq\mathcal{T}}\|\varphi(t)\|^{p}
\leq C_p\big(\|a(0)\|^{p}+1\big)e^{C_p\mathcal{T}}
\leq C_p\big(\|a(0)\|^{p}+1\big)e^{C_pT_0}.
\end{equation*}
As the equation above holds for any radius $n\in\mathcal{N}$ in the definition of the stopping time $\mathcal{T}$ we can pass to the monotone limit to obtain
\begin{equation*}
\mathbb{E}\sup_{0\leq T\leq T_{0}}\|\varphi(T)\|^{p}\leq C_p\big(\|a(0)\|^{p}+1\big).
\end{equation*}
Finally, recall that the norm in $\mathcal{H}^\theta$
and the norm in $\mathcal{H}$ are equivalent on $\mathcal{N}$.
\end{proof}

The next step now is to remove the error from the equation for $a$ to obtain the amplitude equation. We show an error estimate between $a$ and the solution $\varphi$ of the amplitude equation.

\begin{lemma}\label{ab}
Thanks to the use of $(\textbf{A1})$-$(\textbf{A5})$ and $\|b(0)\|_{\theta}\leq\varepsilon^{-\kappa}$, for any $p\in(1,\alpha)$, there exists a constant $C_p>0$ such that
\begin{equation*}
\mathbb{E}\sup_{0\leq T\leq \tau^{*}}\|a(T)-\varphi(T)\|_{\theta}^{p}\leq C_p\varepsilon^{p-18\kappa p}.
\end{equation*}
\end{lemma}
\begin{proof}
For the proof we derive an equation for the error $a-\varphi$ and proceed similarly that for the bound on $\varphi$. But as $R$
(defined in \eqref{R}) is not differentiable in the It\^o-sense, we first substitute $g:=a-R$. Clearly, we have
\begin{equation*}
g(T)=a(0)+\int_{0}^{T}\mathcal{L}_{c}(g(\tau)+R(\tau))d\tau+\int_{0}^{T}\mathcal{F}_{c}(g(\tau)+R(\tau))d\tau+\int_{0}^{T}G'_{c}(0)\cdot (g(\tau)+R(\tau))d\tilde{L}^{\alpha}(\tau).
\end{equation*}
Defining the error $e:=\varphi-g=\varphi-a+R$, we get
\begin{equation*}
e(T)=\int_{0}^{T}\mathcal{L}_{c}e(\tau)d\tau-\int_{0}^{T}\mathcal{L}_{c}R(\tau)d\tau
+\int_{0}^{T}\mathcal{F}_{c}(\varphi(\tau))d\tau-\int_{0}^{T}\mathcal{F}_{c}(\varphi(\tau)-e(\tau)+R(\tau))d\tau
+\int_{0}^{T}G'_{c}(0)\cdot (e(\tau)-R(\tau))d\tilde{L}^{\alpha}(\tau).
\end{equation*}
Let $f$ be the smooth function on $\mathcal{H}$ given by
\begin{equation*}
f(\cdot)=\big(\|\cdot\|^{2}+\delta_\varepsilon\big)^{\frac{p}{2}},\quad \text{where}~~ \delta_\varepsilon=\varepsilon^{2}.
\end{equation*}
For any $x, h\in\mathcal{H}$,
\begin{equation*}
f'(x)h=\frac{p}{\big(\|x\|^{2}+\delta_\varepsilon\big)^{1-\frac{p}{2}}}\langle x,h\rangle
\end{equation*}
and
\begin{equation*}
f''(x)(h,h)=\frac{p}{\big(\|x\|^{2}+\delta_\varepsilon\big)^{1-\frac{p}{2}}}\langle h,h\rangle+\frac{p(p-2)}{\big(\|x\|^{2}+\delta_\varepsilon\big)^{2-\frac{p}{2}}}\langle x,h\rangle\langle x,h\rangle\leq\frac{p(p-1)}{\big(\|x\|^{2}+\delta_\varepsilon\big)^{1-\frac{p}{2}}}\|h\|^{2}.
\end{equation*}
 Applying It\^o's formula to compute
\begin{align}\nonumber
&f(e(T))\\ \nonumber
&=\delta_\varepsilon^{\frac{p}{2}}+\int_{0}^{T}\frac{\langle\mathcal{L}_{c}e(\tau),pe(\tau)\rangle}{\big(\|e(\tau)\|^{2}+\delta_\varepsilon\big)^{1-\frac{p}{2}}}d\tau-\int_{0}^{T}\frac{\langle\mathcal{L}_{c}R(\tau),pe(\tau)\rangle}{\big(\|e(\tau)\|^{2}+\delta_\varepsilon\big)^{1-\frac{p}{2}}}d\tau\\ \nonumber
&~~~+\int_{0}^{T}\frac{\langle\mathcal{F}_{c}(\varphi(\tau))-\mathcal{F}_{c}(\varphi(\tau)-e(\tau)+R(\tau)),pe(\tau)\rangle}{\big(\|e(\tau)\|^{2}+\delta_\varepsilon\big)^{1-\frac{p}{2}}}d\tau\\ \nonumber
&~~~+\sum_{k=1}^{\infty}\int_{0}^{T}\int_{|y|\geq1}[f(e(\tau-)+G'_{c}(0)\cdot (e(\tau)-R(\tau))ye_{k})-f(e(\tau-))]N_{k}(d\tau,dy)\\ \nonumber
&~~~+\sum_{k=1}^{\infty}\int_{0}^{T}\int_{|y|<1}[f(e(\tau-)+G'_{c}(0)\cdot (e(\tau)-R(\tau))ye_{k})-f(e(\tau-))]\tilde{N}_{k}(d\tau,dy)\\ \nonumber
&~~~+\sum_{k=1}^{\infty}\int_{0}^{T}\int_{|y|<1}\Big[f(e(\tau-)+G'_{c}(0)\cdot (e(\tau)-R(\tau))ye_{k})-f(e(\tau-))-\frac{\langle G'_{c}(0)\cdot (e(\tau)-R(\tau))ye_{k},pe(\tau-)\rangle}{\big(\|e(\tau-)\|^{2}+\delta_\varepsilon\big)^{1-\frac{p}{2}}}\Big]\nu_{\alpha}(dy)d\tau.
\end{align}
We derive
\begin{equation}\label{L0}
\int_{0}^{T}\frac{\langle\mathcal{L}_{c}e(\tau),pe(\tau)\rangle}{\big(\|e(\tau)\|^{2}+\delta_\varepsilon\big)^{1-\frac{p}{2}}}d\tau\leq C_{p}\int_{0}^{T}\|e(\tau)\|^{p}d\tau,
\end{equation}
and
\begin{align}\nonumber
	-\int_{0}^{T}\frac{\langle\mathcal{L}_{c}R(\tau),pe(\tau)\rangle}{\big(\|e(\tau)\|^{2}+\delta_\varepsilon\big)^{1-\frac{p}{2}}}d\tau
	&\leq C_{p}\int_{0}^{T}\frac{\|R(\tau)\|\|e(\tau)\|}{\big(\|e(\tau)\|^{2}+\delta_\varepsilon\big)^{1-\frac{p}{2}}}d\tau\leq C_{p}\int_{0}^{T}\|e(\tau)\|^{p-1}\|R(\tau)\|d\tau\\ \label{L}
	&\leq C_{p}\int_{0}^{T}\|e(\tau)\|^{p}d\tau+C_{p}\int_{0}^{T}\|R(\tau)\|^{p}d\tau.
\end{align}
By condition \eqref{fvw}, Young's inequality makes sure that
\begin{align}\nonumber
\int_{0}^{T}\frac{\langle\mathcal{F}_{c}(\varphi(\tau))-\mathcal{F}_{c}(\varphi(\tau)-e(\tau)+R(\tau)),pe(\tau)\rangle}{\big(\|e(\tau)\|^{2}+\delta_\varepsilon\big)^{1-\frac{p}{2}}}d\tau&\leq C_{p}\int_{0}^{T}\frac{\|R(\tau)\|^{4}}{\big(\|e(\tau)\|^{2}+\delta_\varepsilon\big)^{1-\frac{p}{2}}}d\tau+C_{p}\int_{0}^{T}\frac{\|\varphi(\tau)\|^{2}\|R(\tau)\|^{2}}{\big(\|e(\tau)\|^{2}+\delta_\varepsilon\big)^{1-\frac{p}{2}}}d\tau\\\label{L1}
&\leq C_{p}\int_{0}^{T}\|R(\tau)\|^{4}\delta_\varepsilon^{\frac{p}{2}-1}d\tau+C_{p}\int_{0}^{T}\|\varphi(\tau)\|^{2}\|R(\tau)\|^{2}\delta_\varepsilon^{\frac{p}{2}-1}d\tau.
\end{align}

The stochastic term  is bounded as follows. We examine

\begin{align}\nonumber
&\mathbb{E}\sup_{0\leq t\leq\tau^{*}\wedge{T}}\Big|\sum_{k=1}^{\infty}\int_{0}^{t}\int_{|y|\geq1}[f(e(\tau-)+G'_{c}(0)\cdot (e(\tau)-R(\tau))ye_{k})-f(e(\tau-))]N_{k}(d\tau,dy)\Big|\\ \nonumber
&\leq C\sum_{k=1}^{\infty}\mathbb{E}\Big(\int_{0}^{\tau^{*}\wedge{T}}\int_{|y|\geq1}|f(e(\tau)+G'_{c}(0)\cdot (e(\tau)-R(\tau))ye_{k})-f(e(\tau))|N_{k}(d\tau,dy)\Big)\\ \nonumber
&=C\sum_{k=1}^{\infty}\mathbb{E}\Big(\int_{0}^{\tau^{*}\wedge{T}}\int_{|y|\geq1}|f(e(\tau)+G'_{c}(0)\cdot (e(\tau)-R(\tau))ye_{k})-f(e(\tau))|\nu_{\alpha}(dy)d\tau\Big)\\ \nonumber
&\leq C\sum_{k=1}^{\infty}\mathbb{E}\Big(\int_{0}^{\tau^{*}\wedge{T}}\int_{|y|\geq1}\int_{0}^{1}\|f'(e(\tau)+\xi G'_{c}(0)\cdot(e(\tau)-R(\tau))(\tau)ye_{k})\|d\xi\|G'_{c}(0)\cdot (e(\tau)-R(\tau))ye_{k}\|\nu_{\alpha}(dy)d\tau\Big)\\ \nonumber
&\leq C_p\sum_{k=1}^{\infty}\mathbb{E}\Big(\int_{0}^{\tau^{*}\wedge{T}}\int_{|y|\geq1}(\|e(\tau)\|^{p-1}+\| G'_{c}(0)\cdot(e(\tau)-R(\tau))ye_{k}\|^{p-1})\|G'_{c}(0)\cdot (e(\tau)-R(\tau))ye_{k}\|\nu_{\alpha}(dy)d\tau\Big)\\ \nonumber
&\leq C_p\Big[\sum_{k=1}^{\infty}\beta_{k}\int_{|y|\geq1}|y|\nu_{\alpha}(dy)\mathbb{E}\int_{0}^{\tau^{*}\wedge{T}}\|e(\tau)\|^{p-1}\|e(\tau)-R(\tau)\|d\tau+\sum_{k=1}^{\infty}\beta_{k}^{p}\int_{|y|\geq1}|y|^{p}\nu_{\alpha}(dy)\mathbb{E}\int_{0}^{\tau^{*}\wedge{T}}\|e(\tau)-R(\tau)\|^{p}d\tau\Big]\\ \nonumber
&\leq C_p\mathbb{E}\int_{0}^{\tau^{*}\wedge{T}}[\|e(\tau)\|^{p}+\|e(\tau)\|^{p-1}\|R(\tau)\|]d\tau+C_p\mathbb{E}\int_{0}^{\tau^{*}\wedge{T}}[\|e(\tau)\|^{p}+\|R(\tau)\|^{p}]d\tau\\ \label{L2}
&\leq C_p\mathbb{E}\int_{0}^{\tau^{*}\wedge{T}}\|e(\tau)\|^{p}d\tau+C_p\mathbb{E}\int_{0}^{\tau^{*}\wedge{T}}\|R(\tau)\|^{p}d\tau,
\end{align}
where we used Young's inequality in the last inequality.
By Burkholder-Davis-Gundy inequality,
\allowdisplaybreaks
\begin{eqnarray}\nonumber
\lefteqn{\mathbb{E}\sup_{0\leq t\leq\tau^{*}\wedge{T}}\Big|\sum_{k=1}^{\infty}\int_{0}^{t}\int_{|y|<1}[f(e(\tau-)+G'_{c}(0)\cdot (e(\tau)-R(\tau))ye_{k})-f(e(\tau-))]\tilde{N}_{k}(d\tau,dy)\Big|}
\\ \nonumber
&\leq& C\sum_{k=1}^{\infty}\mathbb{E}\Big(\int_{0}^{\tau^{*}\wedge{T}}\int_{|y|<1}|f(e(\tau)+G'_{c}(0)\cdot (e(\tau)-R(\tau))ye_{k})-f(e(\tau))|^{2}
\nu_{\alpha}(dy)d\tau\Big)^{\frac{1}{2}}  \\ \nonumber
&\leq& C\sum_{k=1}^{\infty}\mathbb{E}\Big(\int_{0}^{\tau^{*}\wedge{T}}\int_{|y|<1}\int_{0}^{1}\|f'(e(\tau)+\xi G'_{c}(0)\cdot (e(\tau)-R(\tau))ye_{k})\|^{2}d\xi\|G'_{c}(0)\cdot(e(\tau)-R(\tau))ye_{k}\|^{2}\nu_{\alpha}(dy)d\tau\Big)^{\frac{1}{2}}\\ \nonumber
&\leq& C_p\sum_{k=1}^{\infty}\mathbb{E}\Big(\int_{0}^{\tau^{*}\wedge{T}}\int_{|y|<1}(\|e(\tau)\|^{2p-2}+\| G'_{c}(0)\cdot (e(\tau)-R(\tau))ye_{k}\|^{2p-2})\|G'_{c}(0)\cdot (e(\tau)-R(\tau))ye_{k}\|^{2}\nu_{\alpha}(dy)d\tau\Big)^{\frac{1}{2}}\\ \nonumber
&\leq& C_p\Big(\sum_{k=1}^{\infty}\beta_{k}\Big[\int_{|y|<1}|y|^{2}\nu_{\alpha}(dy)\Big]^{\frac{1}{2}}\mathbb{E}\Big[\int_{0}^{\tau^{*}\wedge{T}}\|e(\tau)\|^{2p-2}\|e(\tau)-R(\tau)\|^{2}d\tau\Big]^{\frac{1}{2}} \\ \nonumber
&& +\sum_{k=1}^{\infty}\beta_{k}^{p}\Big[\int_{|y|<1}|y|^{2p}\nu_{\alpha}(dy)\Big]^{\frac{1}{2}}\mathbb{E}\Big[\int_{0}^{\tau^{*}\wedge{T}}\|e(\tau)-R(\tau)\|^{2p}d\tau\Big]^{\frac{1}{2}}\Big)\\ \nonumber
&\leq& C_p\mathbb{E}\Big[\int_{0}^{\tau^{*}\wedge{T}}[\|e(\tau)\|^{2p}+\|e(\tau)\|^{2p-2}\|R(\tau)\|^{2}]d\tau\Big]^{\frac{1}{2}}+C_p\mathbb{E}\Big[\int_{0}^{\tau^{*}\wedge{T}}[\|e(\tau)\|^{2p}+\|R(\tau)\|^{2p}]d\tau\Big]^{\frac{1}{2}}\\ \nonumber
&\leq& C_p\mathbb{E}\Big[\int_{0}^{\tau^{*}\wedge{T}}\|e(\tau)\|^{2p}d\tau\Big]^{\frac{1}{2}}+C_p\mathbb{E}\Big[\int_{0}^{\tau^{*}\wedge{T}}\|R(\tau)\|^{2p}d\tau\Big]^{\frac{1}{2}}\\ \label{L3}
&\leq& C_p\mathbb{E}\sup_{0\leq T\leq\tau^{*}\wedge{T}}\|e(T)\|^{p}+C_p\int_{0}^{{T}}\mathbb{E}\sup_{0\leq r\leq\tau^{*}\wedge\tau}\|e(r)\|^{p}d\tau+C_p\mathbb{E}\Big[\int_{0}^{\tau^{*}\wedge{T}}\|R(\tau)\|^{2p}d\tau\Big]^{\frac{1}{2}}.
\end{eqnarray}
The Taylor's expansion has been successfully used to monitor
\begin{align}\nonumber
&\mathbb{E}\sup_{0\leq t\leq\tau^{*}\wedge{T}}\Big|\sum_{k=1}^{\infty}\int_{0}^{t}\int_{|y|<1}\Big[f(e(\tau-)+G'_{c}(0)\cdot (e(\tau)-R(\tau))ye_{k})-f(e(\tau-))-\frac{\langle G'_{c}(0)\cdot (e(\tau)-R(\tau))ye_{k},pe(\tau-)\rangle}{\big(\|e(\tau)\|^{2}+\delta_\varepsilon\big)^{1-\frac{p}{2}}}\Big]\nu_{\alpha}(dy)d\tau\Big|\\ \nonumber
&\leq C\sum_{k=1}^{\infty}\mathbb{E}\int_{0}^{\tau^{*}\wedge{T}}\int_{|y|<1}\Big|f(e(\tau)+G'_{c}(0)\cdot (e(\tau)-R(\tau))ye_{k})-f(e(\tau))-\frac{\langle G'_{c}(0)\cdot (e(\tau)-R(\tau))ye_{k},pe(\tau)\rangle}{\big(\|e(\tau)\|^{2}+\delta_\varepsilon\big)^{1-\frac{p}{2}}}\Big|\nu_{\alpha}(dy)d\tau\\ \nonumber
&\leq C_p\sum_{k=1}^{\infty}\mathbb{E}\int_{0}^{\tau^{*}\wedge{T}}\int_{|y|<1}\frac{\|G'_{c}(0)\cdot (e(\tau)-R(\tau))ye_{k}\|^{2}}{\big(\|e(\tau)\|^{2}+\delta_\varepsilon\big)^{1-\frac{p}{2}}}\nu_{\alpha}(dy)d\tau\\ \nonumber
&\leq C_p\sum_{k=1}^{\infty}\beta_{k}^{2}\int_{|y|<1}|y|^{2}\nu_{\alpha}(dy)\mathbb{E}\int_{0}^{\tau^{*}\wedge{T}}\frac{\|e(\tau)-R(\tau)\|^{2}}{\big(\|e(\tau)\|^{2}+\delta_\varepsilon\big)^{1-\frac{p}{2}}}d\tau\\ \nonumber
&\leq C_p\mathbb{E}\int_{0}^{\tau^{*}\wedge{T}}\frac{\|e(\tau)\|^{2}}{\big(\|e(\tau)\|^{2}+\delta_\varepsilon\big)^{1-\frac{p}{2}}}d\tau+C_p\mathbb{E}\int_{0}^{\tau^{*}\wedge{T}}\frac{\|R(\tau)\|^{2}}{\big(\|e(\tau)\|^{2}+\delta_\varepsilon\big)^{1-\frac{p}{2}}}d\tau\\ \label{L4}
&\leq C_p\mathbb{E}\int_{0}^{\tau^{*}\wedge{T}}\|e(\tau)\|^{p}d\tau+C_p\mathbb{E}\int_{0}^{\tau^{*}\wedge{T}}\|R(\tau)\|^{2}\delta_\varepsilon^{\frac{p}{2}-1}d\tau.
\end{align}
Therefore, it should be evident from collecting together \eqref{Re}, \eqref{fine} and \eqref{L0}-\eqref{L4} that
\begin{equation*}
\mathbb{E}\sup_{0\leq t\leq\tau^{*}\wedge{T}}\|e(t)\|^{p}\leq C_p\varepsilon^{p-18\kappa p}+C_p\int_{0}^{{T}}\mathbb{E}\sup_{0\leq r\leq\tau^{*}\wedge\tau}\|e(r)\|^{p}d\tau.
\end{equation*}
The use of Gronwall's lemma allows us to identify
\begin{equation*}
\mathbb{E}\sup_{0\leq T\leq\tau^{*}}\|e(T)\|_{\theta}^{p}\leq C_p\varepsilon^{p-18\kappa p},
\end{equation*}
where we used that the norm in $\mathcal{H}^\theta$
and the norm in $\mathcal{H}$ are equivalent on $\mathcal{N}$.
Moreover, because of \eqref{Re},
\begin{equation*}
\mathbb{E}\sup_{0\leq T\leq\tau^{*}}\|a(T)-\varphi(T)\|_{\theta}^{p}\leq\mathbb{E}\sup_{0\leq T\leq\tau^{*}}\|e(T)\|_{\theta}^{p}+\mathbb{E}\sup_{0\leq T\leq\tau^{*}}\|R(T)\|_{\theta}^{p}\leq C_p\varepsilon^{p-18\kappa p}.
\end{equation*}
\end{proof}
\begin{remark}
It is worthwhile to note that due to Lemma \ref{ab} and \eqref{fine}, for any $p\in(0,\alpha)$ and $\kappa\in(0,\frac{1}{18})$,
\begin{equation}\label{at}
\mathbb{E}\sup_{0\leq T\leq\tau^{*}}\|a(T)\|_{\theta}^{p}\leq\mathbb{E}\sup_{0\leq T\leq\tau^{*}}\|a(T)-\varphi(T)\|_{\theta}^{p}+\mathbb{E}\sup_{0\leq T\leq\tau^{*}}\|\varphi(T)\|_{\theta}^{p}\leq C_p(\|a(0)\|_{\theta}^{p}+1).
\end{equation}
\end{remark}

A remarkable find of Lemma \ref{ab} is the following bound on $\mathcal{R}_c(T)=\varepsilon[a(T)-\varphi(T)]$:
\begin{corollary}
For any $p\in(1,\alpha)$, there exists a constant $C_p>0$ such that
\begin{equation}\label{RR}
\mathbb{E}\sup_{0\leq T\leq\tau^{*}}\|\mathcal{R}_c(T)\|_{\theta}^{p}\leq C_{p}\varepsilon^{2p-18\kappa p}.
\end{equation}
\end{corollary}

As Lemma \ref{IJ} indicates that $I$ and $J$ are uniformly small, but up to now we only verified an $L^p$-bound on $K$. But in order to show that the stopping time $\tau^*$ is large, it is important to  bound
\begin{equation*}
K(T)=\frac{1}{\varepsilon}\int_{0}^{T}e^{\varepsilon^{-2}(T-\tau)\mathcal{A}_{s}}G_{s}(\varepsilon a(\tau)+\varepsilon b(\tau))d\tilde{L}^{\alpha}(\tau)
\end{equation*}
uniformly in time.
The moment inequality \eqref{Lmoment} is unfortunately not available here, as the integrand in $K(T)$ depends on $T$. We will use again the Riesz-Nagy-trick.

\begin{lemma}\label{M}
Assume the setting of Lemma \ref{ab}. For any $p\in(0,\alpha)$, there exists a constant $C_p>0$ such that
\begin{equation}\label{MT}
\mathbb{E}\sup_{0\leq T\leq\tau^{*}}\|K(T)\|_{\theta}^{p}\leq C_{p}\varepsilon^{-\kappa p}.
\end{equation}
\end{lemma}
\begin{proof}
There exists a Hilbert space $\tilde{\mathcal{H}}$ and a unitary strongly continuous group $\{U(t)\}_{t\in\mathbb{R}}$ on $\tilde{\mathcal{H}}$ such that $\mathcal{H}$ embeds isometrically into $\tilde{\mathcal{H}}$ and the contraction semigroup $e^{t\mathcal{A}}$ on $\mathcal{H}$ is a projection of $U(t)$, i.e., $\mathcal{P}U(t)=e^{t\mathcal{A}}$ on $\mathcal{H}$ for all $t\geq0$, $\mathcal{P}$ being the orthogonal projection from $\tilde{\mathcal{H}}$ onto $\mathcal{H}$. The moment inequality \eqref{Lmoment} and Riesz-Nagy theorem recognize that
\begin{align*}
\mathbb{E}\sup_{0\leq T\leq\tau^{*}}\|K(T)\|_{\theta}^{p}&\leq\mathbb{E}\sup_{0\leq T\leq T_{0}}\|\frac{1}{\varepsilon}\int_{0}^{T}e^{\varepsilon^{-2}(T-\tau)\mathcal{A}_{s}}\mathds{1}_{[0, \tau^{*}]}(\tau)G_{s}(\varepsilon a(\tau)+\varepsilon b(\tau))d\tilde{L}^{\alpha}(\tau)\|_{\theta}^{p}\\
&=\mathbb{E}\sup_{0\leq T\leq T_{0}}\|\frac{1}{\varepsilon}\int_{0}^{T}\mathcal{P}U_\varepsilon(T)U_\varepsilon(-\tau)\mathds{1}_{[0, \tau^{*}]}(\tau)G_{s}(\varepsilon a(\tau)+\varepsilon b(\tau))d\tilde{L}^{\alpha}(\tau)\|_{\theta}^{p}\\
 &\leq C_{p}\mathbb{E}\Big(\frac{1}{\varepsilon^{\alpha}}\int_{0}^{T_{0}}\|U_\varepsilon(-\tau)\mathds{1}_{[0, \tau^{*}]}(\tau)G_{s}(\varepsilon a(\tau)+\varepsilon b(\tau))\|_{L_{HS}}^{\alpha}d\tau\Big)^{\frac{p}{\alpha}}\\
&\leq C_{p}\mathbb{E}\Big(\int_{0}^{T_{0}}\mathds{1}_{[0, \tau^{*}]}(\tau)\|a(\tau)+ b(\tau)\|_{\theta}^{\alpha}d\tau\Big)^{\frac{p}{\alpha}}\\
 &\leq C_{\alpha,p}\mathbb{E}\Big[\int_{0}^{T_{0}}\mathds{1}_{[0, \tau^{*}]}(\tau)\|a(\tau)\|_{\theta}^{\alpha}d\tau\Big]^{\frac{p}{\alpha}}+C_{\alpha,p}\mathbb{E}\Big[\int_{0}^{T_{0}}\mathds{1}_{[0, \tau^{*}]}(\tau)\|b(\tau)\|_{\theta}^{\alpha}d\tau\Big]^{\frac{p}{\alpha}}.
\end{align*}
Notice that when $\tau\leq\tau^{*}$, the first term on the right-hand side is bounded by $C\varepsilon^{-\kappa p}$ uniformly in time. Consequently, we observe from Corollary \ref{Co} that
\begin{equation*}
\mathbb{E}\sup_{0\leq T\leq\tau^{*}}\|K(T)\|_{\theta}^{p}\leq C_{p}\varepsilon^{-\kappa p}.
\end{equation*}
\end{proof}

Furthermore, by the definition $b(T)=Q(T)+I(T)+J(T)+K(T)$, Remark \ref{Q}, Lemma \ref{IJ} and Lemma \ref{M} enable us to gain a detailed understanding of the bound on $b$.
\begin{lemma}\label{B}
With assumptions $(\textbf{A1})$-$(\textbf{A5})$ and $\|b(0)\|_{\theta}\leq\varepsilon^{-\kappa}$, for any $p\in(0,\alpha)$, there exists a constant $C_p>0$ such that
\begin{equation}\label{BB}
\mathbb{E}\sup_{0\leq T\leq\tau^{*}}\|b(T)\|_{\theta}^{p}\leq C_{p}\varepsilon^{-\kappa p}.
\end{equation}
\end{lemma}

Before continuing, we construct a subset of $\Omega$, which enjoys nearly full probability.

\begin{definition}\label{taoxin}
For $\kappa\in(0,\frac{1}{18})$, from the definition of $\tau^{*}$ as in \eqref{tao}, define the set $\Omega^{*}\subset\Omega$ of all $\omega\in\Omega$ such that all these estimates
\begin{equation*}
\sup_{0\leq T\leq\tau^{*}}\|a(T)\|_{\theta}<\varepsilon^{-\kappa},~~~\sup_{0\leq T\leq\tau^{*}}\|b(T)\|_{\theta}<\varepsilon^{-2\kappa},~~~\sup_{0\leq T\leq\tau^{*}}\|\mathcal{R}_c(T)\|_{\theta}<\varepsilon^{2-19\kappa}~~~\text{and}~~~\|\mathcal{R}_s(T)\|_{L^p([0,\tau^{*}];\mathcal{H}^{\theta})}<\varepsilon^{3-7\kappa }
\end{equation*}
hold.
\end{definition}
\begin{lemma}
The set $\Omega^{*}$ in Definition \ref{taoxin} has approximately probability $1$.
\end{lemma}
\begin{proof}
It is natural to consider
\begin{align*}
\mathbb{P}(\Omega^{*})
\geq & 1-\mathbb{P}(\sup_{0\leq T\leq\tau^{*}}\|a(T)\|_{\theta}\geq\varepsilon^{-\kappa})-\mathbb{P}(\sup_{0\leq T\leq\tau^{*}}\|b(T)\|_{\theta}\geq\varepsilon^{-2\kappa})
\\ &
-\mathbb{P}(\sup_{0\leq T\leq\tau^{*}}\|\mathcal{R}_c(T)\|_{\theta}\geq\varepsilon^{2-19\kappa})-\mathbb{P}(\|\mathcal{R}_s(T)\|_{L^p([0,\tau^{*}];\mathcal{H}^{\theta})}>\varepsilon^{3-7\kappa }).
\end{align*}
Admittedly, Chebychev's inequality, Corollary \ref{Co}  and \eqref{at}-\eqref{BB} illustrate
\begin{align*}
\mathbb{P}(\Omega^{*})
\geq &
1-(\varepsilon^{-\kappa})^{-q}\mathbb{E}\sup_{0\leq T\leq\tau^{*}}\|a(T)\|_{\theta}^{q}
-(\varepsilon^{-2\kappa})^{-q}\mathbb{E}\sup_{0\leq T\leq\tau^{*}}\|b(T)\|_{\theta}^{q}
\\ & -(\varepsilon^{2-19\kappa})^{-q}\mathbb{E}\sup_{0\leq T\leq\tau^{*}}\|\mathcal{R}_c(T)\|_{\theta}^{q}-(\varepsilon^{3-7\kappa})^{-q}\|\mathcal{R}_s(T)\|_{L^p([0,\tau^{*}];\mathcal{H}^{\theta})}^{q}
\\
\geq & 1
-C\varepsilon^{\kappa q}\varepsilon^{-\frac{1}{2}\kappa q}-C\varepsilon^{2\kappa q}\varepsilon^{-\kappa q}
-C\varepsilon^{(19\kappa-2)q}\varepsilon^{(2-18\kappa)q}-C\varepsilon^{(7\kappa-3)q}\varepsilon^{(3-6\kappa)q}
\\ & \underset{\varepsilon\rightarrow0}{\longrightarrow}1.
\end{align*}
\end{proof}
Lastly, at the level of the present considerations it is relevant to point out that $\|a\|<\varepsilon^{-\kappa}$ on $[0,T_{0}]$ with probability almost $1$. Let us finally prove our main theorem.

\begin{proof}[Proof of Theorem \ref{MR}]
As the definition of $\Omega^{*}$ and $\tau^{*}$ suggest
\begin{equation*}
\Omega^{*}\subseteq\Big\{\sup_{0\leq T\leq\tau^{*}}\|a(T)\|_{\theta}<\varepsilon^{-\kappa},~~\sup_{0\leq T\leq\tau^{*}}\|b(T)\|_{\theta}<\varepsilon^{-2\kappa}\Big\}\subseteq\{\tau^{*}=T_{0}\}\subseteq\Omega.
\end{equation*}
This permits us to obtain on $\Omega^{*}$ that
\begin{equation*}
\sup_{0\leq T\leq T_{0}}\|\mathcal{R}_c(T)\|_{\theta}=\sup_{0\leq T\leq\tau^{*}}\|\mathcal{R}_c(T)\|_{\theta}<\varepsilon^{2-19\kappa}~~~\text{and}~~~\|\mathcal{R}_s(T)\|_{L^p([0,T_0];\mathcal{H}^{\theta})}=\|\mathcal{R}_s(T)\|_{L^p([0,\tau^{*}];\mathcal{H}^{\theta})}< \varepsilon^{3-7\kappa },
\end{equation*}
such that
\begin{equation*}
\mathbb{P}\big(\sup_{0\leq T\leq T_{0}}\|\mathcal{R}_c(T)\|_{\theta}\geq\varepsilon^{2-19\kappa}\big)\leq1-\mathbb{P}(\Omega^{*})\underset{\varepsilon\rightarrow0}{\longrightarrow}0~~~\text{and}~~~\mathbb{P}\big(\|\mathcal{R}_s(T)\|_{L^p([0,T_0];\mathcal{H}^{\theta})}\geq\varepsilon^{3-7\kappa }\big)\leq1-\mathbb{P}(\Omega^{*})\underset{\varepsilon\rightarrow0}{\longrightarrow}0,
\end{equation*}
respectively. Recalling representation \eqref{huaR} of $\mathcal{R}$, the proof is finished.
\end{proof}
\begin{remark}
The uniqueness in the previous theorem should be understood by choosing a version of the solution, i.e., by changing it on null sets.
\end{remark}


\section{Examples and applications}
In this section, we provide two examples to corroborate our analytical results.
The first example is the Ginzburg-Landau equation,
which is an effective amplitude equation for the description of pattern forming systems close to the first instability.
Consider the following stochastic Allen-Cahn equation (real Ginzburg-Landau equation), subject to Dirichlet boundary conditions,
with linear multiplicative noise on the domain $D=[0,\pi]$ of the type
\begin{equation}\label{you}
\partial_{t}u(t)=(\partial_{x}^{2}+1)u(t)+\gamma\varepsilon^{2}u(t)-u^{3}(t)+\varepsilon^{\frac{2}{\alpha}}u(t)\partial_{t}Q^{1/2}L^{\alpha}(t),
\end{equation}
where $\mathcal{A}:=\partial_{x}^{2}+1$, $\mathcal{L}:=\gamma\mathcal{I}$, $\mathcal{F}(u):=-u^{3}$, and $G(u)= u Q^{1/2}$, the multiplication operator combined with the covariance operator $Q$ defined below.

We scale the linear term to be close to bifurcation by choosing $\gamma\varepsilon^{2}$ and choose a noise strength of order $\varepsilon^{\frac{2}{\alpha}}$ with the index of stability $\alpha\in(1,2)$. In this case both the noise and the linear (in)stability will survive in the amplitude equation.

The noise $L^{\alpha}(t)$ is a cylidrical $\alpha$-stable L\'evy process on some stochastic basis $(\Omega,\mathcal{F},\{\mathcal{F}_t\}_{t\in\mathbb{R}},\mathbb{P})$ defined via
\begin{equation*}
L^{\alpha}(t)=\sum_{k=1}^{\infty} L_{k}^{\alpha}(t)e_{k},~~~~~~t\geq0.
\end{equation*}
and the covariance operator $Q$ is defined by $Qe_k=\beta^2_k e_k$ where  $\{\beta_{k}\}_{k=1}^{\infty}$ is one given sequence of positive numbers satisfying
\[
 \sum_{k=1}^\infty \beta_k <\infty\quad
 \text{ and }\quad
 \sum_{k=1}^\infty \beta_k^2 k^2 <\infty.
\]
Let $\mathcal{H}=L^{2}([0,\pi])$ be the Hilbert space of all square integrable real-valued functions defined on the interval $[0,\pi]$ we are going to work in.
Denote $H^1_{0}([0,\pi])$ as the Sobolev space of functions with square integrable derivatives that satisfy Dirichlet boundary conditions.

The existence and uniqueness of global mild solutions (i.e., $\tau_{ex}=\infty$) for equation \eqref{you} based on a Galerkin approximation is standard, so we won't go into detail here.

The eigenvalues of $-\mathcal{A}=-\partial_{x}^{2}-1$ are calculated accurately to be $\lambda_{k}=k^{2}-1$, $k=1,2,...$, and then $\lambda_{k} \to \infty$ for $k\to \infty$. The associated eigenvectors are $e_{k}(x)=({2}{\pi}^{-1})^{1/2}\sin(kx)$
and the dominant space is $\mathcal{N}=\text{span}\{e_{1}\}$.
Hence assumptions $(\textbf{A1})$ is valid.

Arguably, assumption $(\textbf{A2})$ is true for example for any $\theta>\frac{1}{2}$
and $\sigma=0$. As for the norm in $\mathcal{H}^{\theta/2}=H_0^\theta$, we then have
$\|uv\|_{H_0^\theta}\leq C\|u\|_{\theta}\|v\|_{H_0^\theta}$.
For simplicity, we will fix $\theta=1$.

It should be noted, that on the one-dimensional space $\mathcal{N}$
the $H_0^{\theta}$-norm is just a multiple of the $\mathcal{H}$-norm.
Since $\mathcal{F}$ is a standard cubic nonlinearity, for $u,w\in\mathcal{N}$,
\begin{equation*}
\langle\mathcal{F}_{c}(u),u\rangle=-\int_{0}^{\pi}u^{4}(x)dx\leq0,\,\,\,\,\langle\mathcal{F}_{c}(u,u,w),w\rangle=-\int_{0}^{\pi}u^{2}(x)\omega^{2}(x) dx\leq0,
\end{equation*}
and condition \eqref{fvw} holds for some positive constants $C_{0}$, $C_{1}$ and $C_{2}$, and thus assumption $(\textbf{A3})$ is true.

The Hilbert-Schmidt operator
$G:\mathcal{H}^{1}\rightarrow L_{HS}(\mathcal{H},\mathcal{H}^{1})$
satisfies $G(0)=0$
and
\begin{equation*}
\|G(u)\|_{L_{HS}}^2
=\sum_{k=1}^{\infty}\|u\cdot Q^{1/2} e_{k}\|_{H^1_0}^{2}
\leq C\sum_{k=1}^{\infty} \beta_k^2 \|u\|_{H^1_0}^{2}\|e_{k}\|_{H^1_0}^{2}
\leq C\sum_{k=1}^{\infty} k^2\beta_k^2 \|u\|_{H^1_0}^{2} <\infty.
\end{equation*}
In addition, $G'(u)\cdot v= Q^{1/2}v$ and $G''(u)=0$.
So assumptions $(\textbf{A4})$ and  $(\textbf{A5})$ follow immediately.

Under our main assumptions, the stochastic Allen-Cahn equation \eqref{you} is well approximated by the amplitude equation
\begin{equation*}
\partial_{T}\varphi(T)=\mathcal{L}_{c}\varphi(T)+\mathcal{F}_{c}(\varphi(T))+[G'_{c}(0)\cdot \varphi(T)]\partial_{T}\tilde{L}^{\alpha}(T)\\
              =\mathcal{P}_{c}[\gamma\varphi(T)-\varphi^{3}(T)]+\mathcal{P}_{c}\varphi(T)Q^{1/2}\partial_{T}\tilde{L}^{\alpha}(T),
\end{equation*}
where $\varphi\in\mathcal{N}$
is determined by the rescaled solution $u(t,x)\approx \varepsilon\varphi(\varepsilon^{2}t, x)$ of \eqref{you}.

To be more specific, we calculate the amplitude equation for the actual amplitude of
$\varphi=\phi\sin(\cdot)$.
We have
$\mathcal{P}_{c}\varphi=\frac{2}{\pi}\int_{0}^{\pi}\phi\sin^{2}(y)dy\sin(\cdot)=\phi\sin(\cdot)$
and
$\mathcal{P}_{c}\mathcal{F}(\varphi)=-\phi^{3}\frac{2}{\pi}\int_{0}^{\pi}\sin(y)\sin^{3}(y)dy\sin(\cdot)=-\frac{3}{4}\phi^{3}\sin(\cdot)$.
Moreover,
\begin{equation*}
\mathcal{P}_{c}\varphi(T)Q^{1/2}\partial_{T}\tilde{L}^{\alpha}(T)
=\phi(T)\sum_{k=1}^{\infty}\beta_k \mathcal{P}_{c}[\sin(\cdot)e_{k}]\partial_{T}\tilde{L}_{k}^{\alpha}(T)
=\phi(T)\sum_{k=1}^{\infty}\delta_{k}\beta_k\partial_{T}\tilde{L}_{k}^{\alpha}(T)\sin(\cdot),
\end{equation*}
where
$\mathcal{P}_{c}[\sin(x)e_{k}(x)]=\mathcal{P}_{c}[\sqrt{\frac{2}{\pi}}\sin(x)\sin(kx)]=(\frac{2}{\pi})^{\frac{3}{2}}\int_{0}^{\pi}\sin^{2}(y)\sin(ky)dy\sin(\cdot)=\delta_{k}\sin(\cdot)$
with
\begin{eqnarray*}
\delta_{k}:=\left\{\begin{array}{l}
(\frac{2}{\pi})^{\frac{3}{2}}\frac{2(\cos(k\pi)-1)}{k(k^{2}-4)},\,\,\,k\neq2,\\
0,\,\,\,\,\,\,\,\,\,\,\,\,\,\,\,\,\,\,\,\,\,\,\,\,\,\,\,\,\,\,\,\,\,\,\,k=2.
 \end{array}\right.
\end{eqnarray*}
Hence the amplitude equation for \eqref{you} is
\begin{equation}\label{amp}
\partial_{T}\phi(T)
=[\gamma\phi(T)-\frac{3}{4}\phi^{3}(T)]+\phi(T)\partial_{T} \sum_{k=1}^{\infty}\beta_k\delta_{k}\tilde{L}_{k}^{\alpha}(T),
\end{equation}
where the deterministic part describes a forward-pitchfork bifurcation. The deterministic counterpart $\dot{\phi}=\gamma\phi-\frac{3}{4}\phi^{3}$ has either one or three fixed points depending on the value of the parameter $\gamma$. When $\gamma\leq0$, there is one stable fixed point at $\phi=0$.
When $\gamma>0$, there are three fixed points at $\phi=0$, $2\sqrt{\gamma/3}$ and $-2\sqrt{\gamma/3}$.

It is remarkable that in the driving L\'evy process all infinitely many one-dimensional $\alpha$-stable L\'evy processes $L_k^{\alpha}$ contribute to the noise in the amplitude equation.
By means of Monte Carlo Simulation, several trajectories of stochastic system \eqref{amp} approach the unique stable equilibrium state $\phi=0$ with the negative bifurcation parameter $\gamma=-0.05$ and the noise intensity $\delta=0.1$. The decrease of the index of stability from $\alpha=1.9$ to $\alpha=1.1$  leads to the increase of the number of the big jumps, as shown in Figure \ref{Fig2}. According to Figure \ref{Fig3},  several trajectories of stochastic system \eqref{amp} close to the two stable equilibrium states near $\phi=1$ and $\phi=-1$ for  the bifurcation parameter $\gamma=0.4$ and the noise intensity $\delta=0.05$. When  the index of stability varies from $\alpha=1.8$ to $\alpha=1.2$, the number of the big jumps increases with decreasing $\alpha$. If we compare  Figure \ref{Fig2} and Figure \ref{Fig3}, stochastic system \eqref{amp} has two stable equilibrium states and one unstable equilibrium state for $\gamma=0.4$ , and only one stable equilibrium state for $\gamma=-0.05$, the change of the number and the stability of equilibrium states exhibits an interesting stochastic bifurcation phenomenon.


\begin{figure}
\begin{center}
\begin{minipage}{3.2in}
\leftline{(i)}
\includegraphics[width=3.2in]{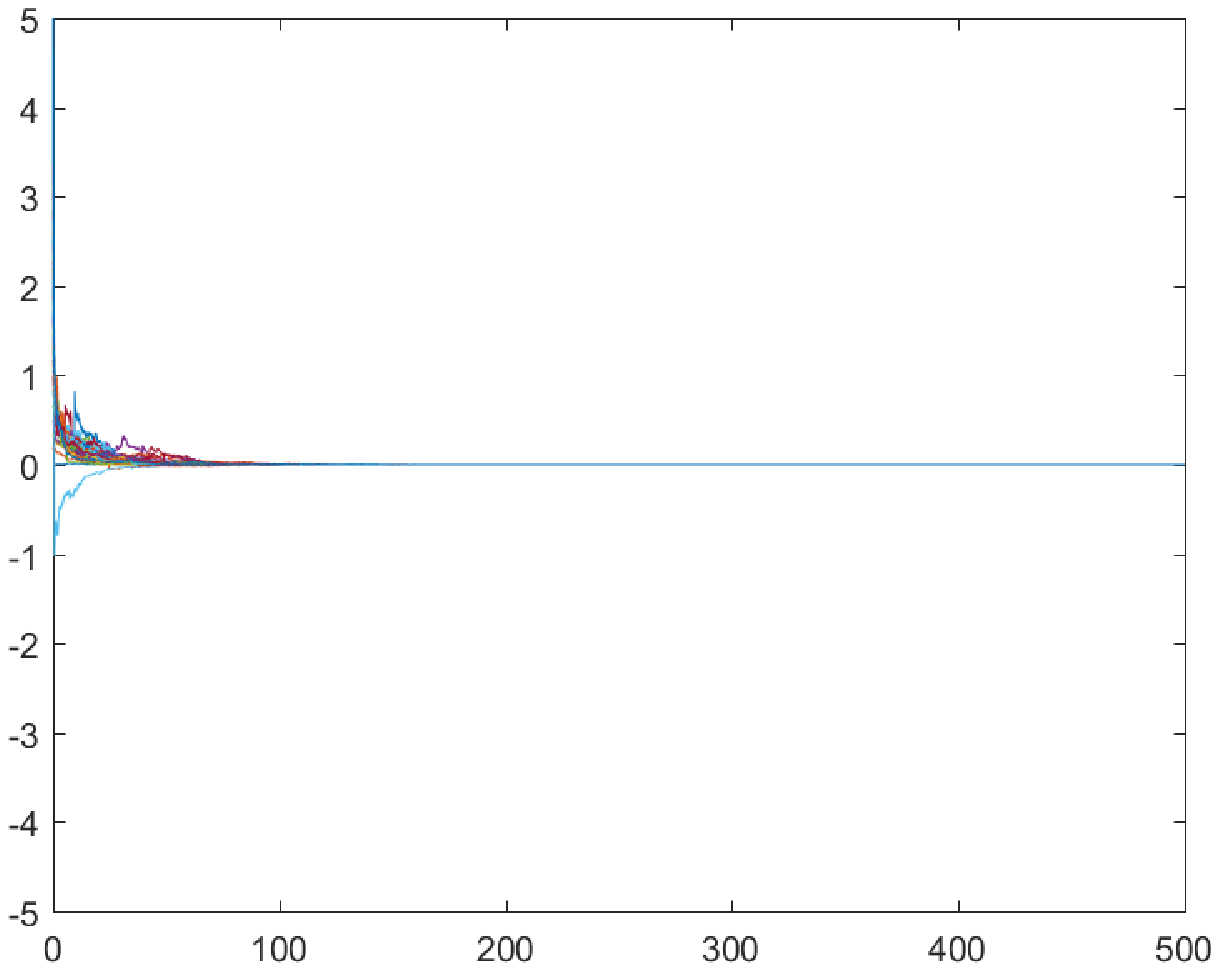}
\end{minipage}
\hfill
  \begin{minipage}{3.2in}
\leftline{(ii)}
\includegraphics[width=3.2in]{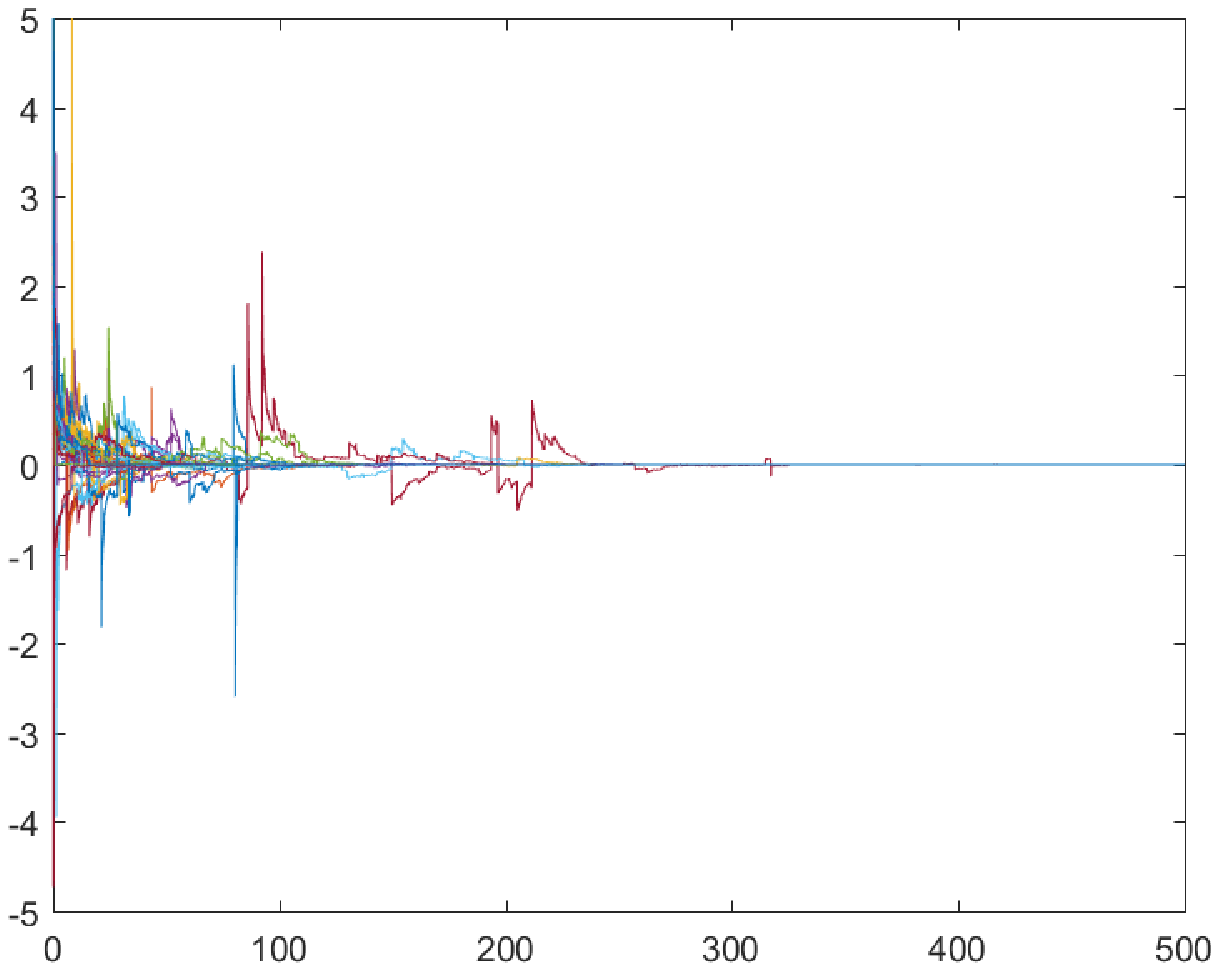}
\end{minipage}
\caption{Use Monte Carlo Simulation to  simulate 50 trajectories of stochastic system \eqref{amp} for the bifurcation parameter $\gamma=-0.05$ and the noise intensity $\delta=0.1$: (i)  the index of stability $\alpha=1.9$; (ii) the index of stability $\alpha=1.1$.
For small $\alpha$ we also see large jumps that lead to large error terms in the estimates.}\label{Fig2}
\end{center}
\end{figure}

\begin{figure}
\begin{center}
\begin{minipage}{3.2in}
\leftline{(a)}
\includegraphics[width=3.2in]{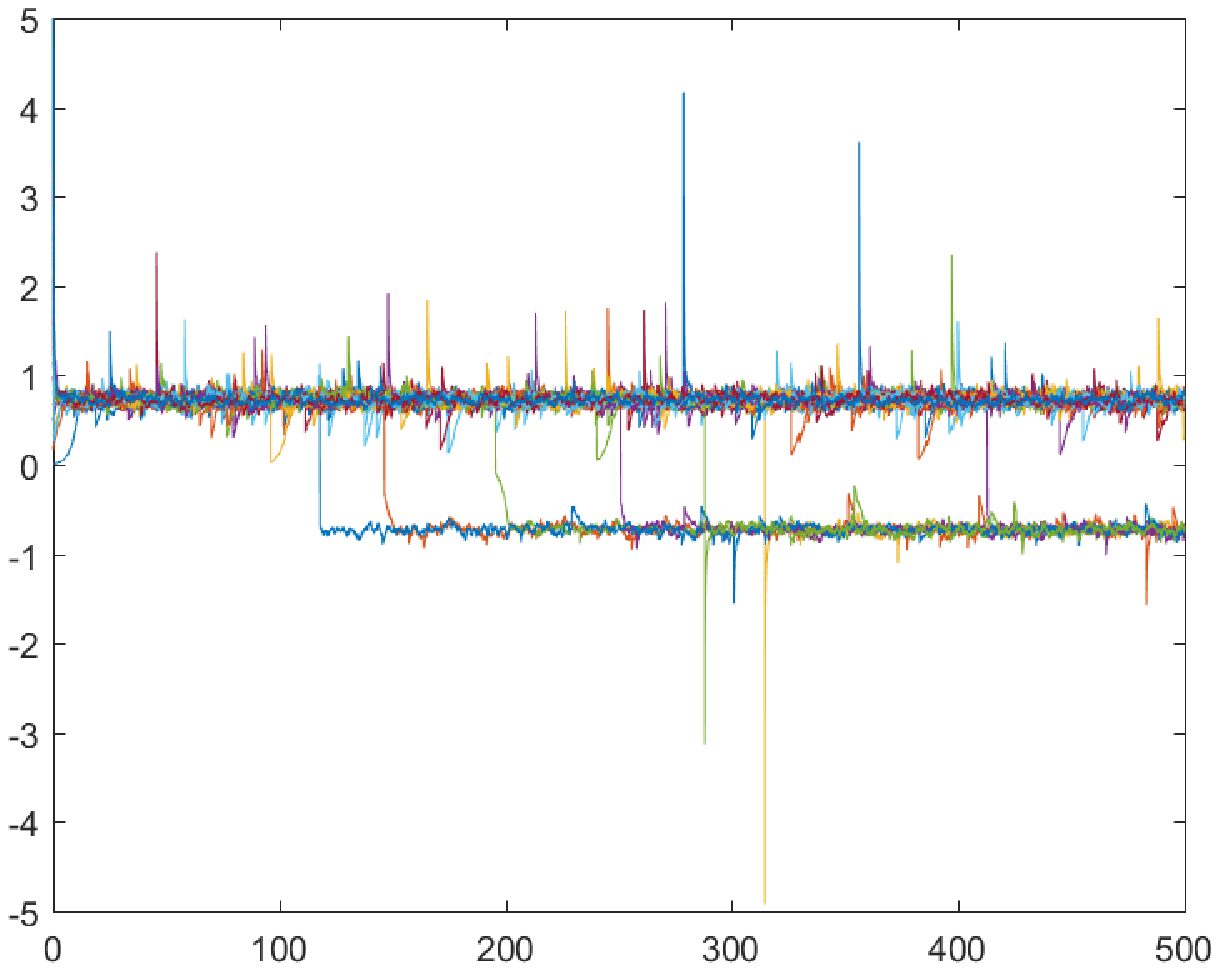}
\end{minipage}
\hfill
  \begin{minipage}{3.2in}
\leftline{(b)}
\includegraphics[width=3.2in]{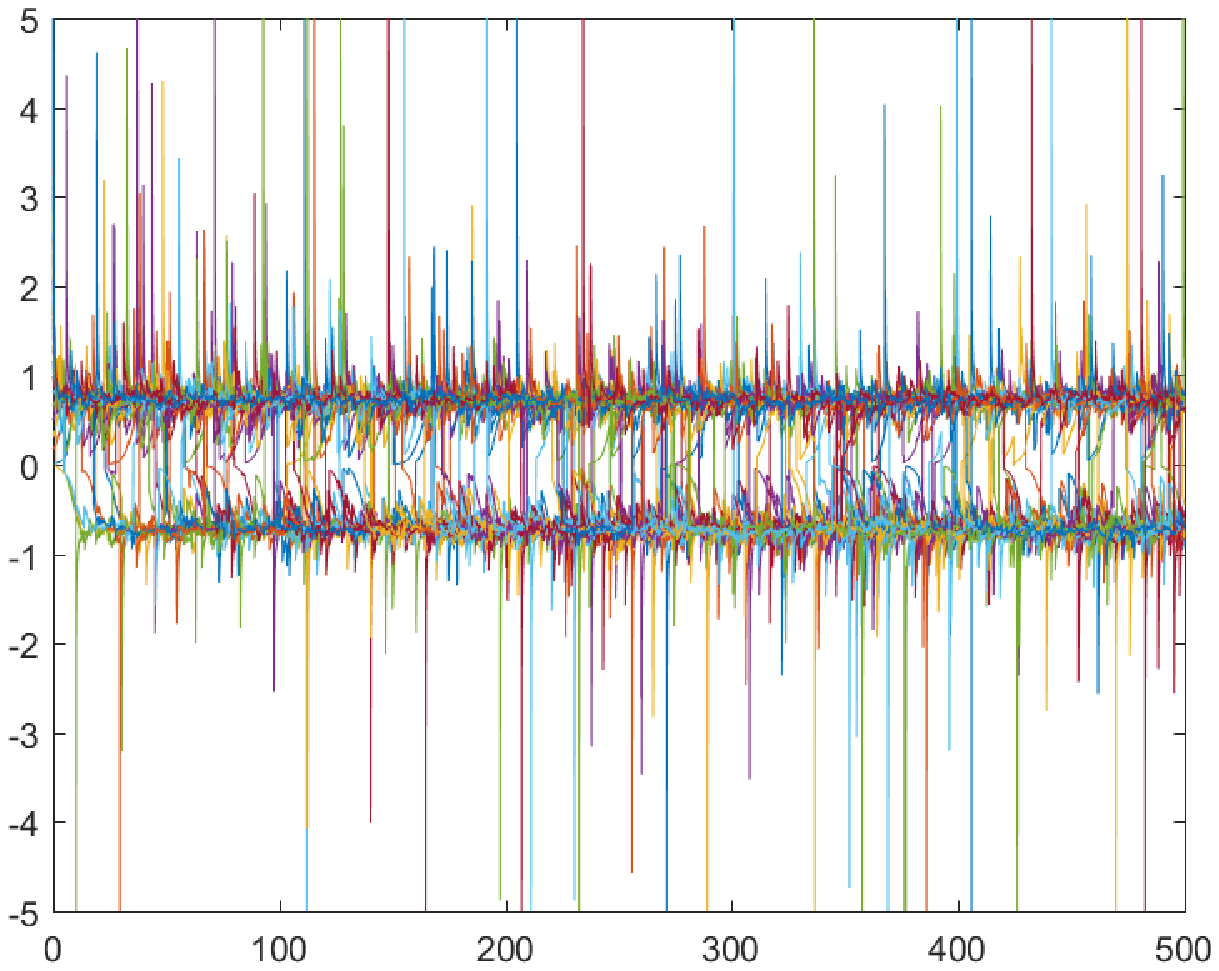}
\end{minipage}
\caption{ Use Monte Carlo Simulation to  simulate 50 trajectories of stochastic system \eqref{amp} for the bifurcation parameter $\gamma=0.4$ and the noise intensity $\delta=0.05$: (a)  the index of stability $\alpha=1.8$; (b) the index of stability $\alpha=1.2$.
Again, we see the large jumps that lead to large error terms that cannot be controlled uniformly in time. }\label{Fig3}
\end{center}
\end{figure}

The second example, which we will discuss very briefly,
is the following surface growth model
\begin{equation}\label{Xu}
\partial_{t}u(t)=-\Delta^{2}u(t)-\mu\Delta u(t)+\nabla\cdot(|\nabla u(t)|^{2}\nabla u(t))+\varepsilon^{\frac{2}{\alpha}}u(t)\partial_{t}Q^{1/2}L^{\alpha}(t),
\end{equation}
subject to periodic boundary conditions on the interval $[0,2\pi]$. In order to get close to the change of stability, we consider $\mu=1+\varepsilon^{2}\gamma$. Therefore,
\begin{equation*}
\mathcal{A}=-\Delta^{2}-\Delta,\,\,\,\,\mathcal{L}=-\gamma\Delta\,\,\,\,\text{and}\,\,\,\,\mathcal{F}(u)=\nabla\cdot(|\nabla u|^{2}\nabla u).
\end{equation*}
One can check that assumptions $(\textbf{A1})$-$(\textbf{A5})$ are satisfied.
The eigenvalues of $-\mathcal{A}=\Delta^{2}+\Delta$ are $\lambda_{k}=k^{4}-k^{2}$, $k=1,2,...$, and then $\lim_{k\to \infty}\lambda_{k}=\infty$.
Consider for $k\in\mathbb{N}$ the eigenfunctions
$e_{k}(x)=\frac{1}{\sqrt{\pi}}\sin(kx)$ and $e_{-k}(x)=\frac{1}{\sqrt{\pi}}\cos(kx)$,
and $e_0(x)=1/\sqrt{2\pi}$.
We obtain $\mathcal{N}=\text{span}\{e_{1},\ e_{-1}\}$
and we will work in the space $\mathcal{H}=L^{2}([0,2\pi])$. And the space $\mathcal{H}^{\theta/4}$ is the standard Sobolev space $H^\theta_{per}$.

Furthermore, if $u=\gamma_{1}\sin+\gamma_2\cos \in\mathcal{N}$, then $\mathcal{F}_{c}(u)=-\frac{3}{4}|\gamma|^{2}(\gamma_1 \sin, \gamma_2\cos)$
and $\langle\mathcal{F}_{c}(u),u\rangle=-\frac{3\pi}{4}|\gamma|^{4}\leq0$.
Moreover, for $\theta=\sigma$,
\begin{equation*}
\|\mathcal{F}(u)\|_{L^{2}}=\|\partial_{x}(\partial_{x}u)^{3}\|_{L^{2}}\leq C\|(\partial_{x}u)^{3}\|_{1}\leq C\|\partial_{x}u\|_{H^1}^{3}\leq C\|u\|_{H^2}^{3}.
\end{equation*}
As before,  $L^{\alpha}$ is a cylindrical L\'evy process and the covariance operator is
$Q^{1/2}$ with $Qe_k=\beta_k^2e_k$.
We suppose $\sum_k\beta_k <\infty$ and $\sum_k\beta_k^2 k^4 <\infty$, so that the assumptions on $G$ are satisfied.

Utilizing the approximation $u(t,x) \approx \varepsilon\varphi(\varepsilon^{2}t, x)$
of \eqref{Xu}, the amplitude equation takes the form
\begin{equation*}
\partial_{T}\varphi(T)=\mathcal{L}_{c}\varphi(T)+\mathcal{F}_{c}(\varphi(T))+[G'_{c}(0)\cdot \varphi(T)]\partial_{T}\tilde{L}^{\alpha}(T)\\
              =\mathcal{P}_{c}[-\mu\Delta\varphi(T)+\nabla\cdot(|\nabla \varphi(T)|^{2}\nabla \varphi(T))]+\mathcal{P}_{c}\varphi(T)\partial_{T}Q^{1/2}\tilde{L}^{\alpha}(T).
\end{equation*}

Suppose that $\phi(T)= (\gamma_1(T)\sin,\gamma_2(T)\cos)$
we can reduce the previous system to ($j=1,2$)
\[
d\gamma_j = [\mu\gamma_j - \frac34 |\gamma|^2 \gamma_j]dT  + \gamma_j d \hat{Z}^\alpha_j
\]
where the driving $\alpha$-stable L\'evy process $\hat{Z}^\alpha$ depends only on $L^\alpha_0,\ L^\alpha_2,\ L^\alpha_{-2}$.


\section{Conclusions and future challenges}
In this work, we analysed a class of stochastic partial differential equations of the form \eqref{GinzburgLandau}
driven by cylindrical $\alpha$-stable L\'evy processes with $\alpha\in(1,2)$ in fractional Sobolev spaces. By utilizing a separation of time-scales,
we explored the  dynamics of the solution $u(t)$ to equation \eqref{GinzburgLandau} on the natural slow time-scale of order $\varepsilon^{-2}$ in the limit $\varepsilon\rightarrow0$.
Here $(\ref{GinzburgLandau})$ was reduced to slow dynamics
on a dominant pattern coupled to dynamics on a fast time scale, which provided an effective tool for the qualitative analysis of the dynamical behaviors.
Our main result in Theorem \ref{MR} stated that near the change of stability, i.e., for small $\varepsilon>0$, the dynamics of \eqref{GinzburgLandau} is well approximated by the amplitude equation \eqref{v} under appropriate conditions.
In order to obtain the error estimates of the approximation result, we introduced the moment inequality \eqref{Lmoment}. The accuracy for those estimations were quantified by $p$-moment with $p\in(0,\alpha)$. The amplitude equation offered a benefit of dimension reduction in characterizing the qualitative properties of stochastic dynamics and detecting rigorously the stochastic bifurcation.

 Let us comment here briefly on possible extensions of those results. We focused on the case of infinite dimensional multiplicative noise with $G(0)=0$. It should be a straightforward modification of our
 analysis to treat additive noise
 or noise with $G(0)\not=0$.
 We have a slighly different scaling of the noise in that case,
 but the general result will be similar.
 The L\'evy noise $L^{\alpha}$ was assumed to be $\alpha$-stable and symmetric in this paper. Quantifying SPDEs driven by general L\'evy processes and  examining the impact of noise on system's dynamics would be of particular useful for scientific computation and further analysis. Moreover, it would be interesting to extend the present considerations to SPDEs on unbounded domains that are intensely studied over the past few years, cf. \cite{BB,BBS}. Such studies are currently in progress and will be reported in future publications.
 \medskip

 \textbf{DATA AVAILABILITY}

The data that support the findings of this study are openly
available in GitHub, Ref. \cite{Y21}.

\medskip
{\it Acknowledgements.}\, The authors are happy to thank Guido Schneider, Haitao Xu and Jinqiao Duan for fruitful discussions on dynamical
systems and stochastic differential equations driven by L\'evy motions. Moreover, we thank Markus Riedle for pointing out reference \cite{Rnew}.






\renewcommand{\theequation}{\thesection.\arabic{equation}}
\setcounter{equation}{0}

\section*{References}


\end{document}